\documentclass[a4paper, 11pt]{article}

\usepackage{fullpage,amsmath,amssymb,amsthm,bm,url}
\usepackage[shortlabels]{enumitem}
\usepackage[hyphenbreaks]{breakurl}

\newcommand{\eps}{\varepsilon}
\newcommand*{\prob}[1]{\mathbb P(#1)}
\newcommand*{\probb}[1]{\mathbb P\bigl(#1\bigr)}
\newcommand*{\mean}[1]{\mathbb E(#1)}
\newcommand{\ee}{\mathrm{e}}

\newcommand*{\floorfrac}[2]{\genfrac{\lfloor}{\rfloor}{}{}{#1}{#2}}

\newcommand*{\floor}[1]{\lfloor #1\rfloor}
\newcommand*{\ceil}[1]{\lceil #1\rceil}
\newcommand{\ie}{i.e.\ }
\newcommand{\Po}{\operatorname{Po}}
\newcommand{\Var}{\operatorname{Var}}

\newtheorem{thm}{Theorem}[section]
\newtheorem{lem}[thm]{Lemma}

\newtheorem{prop}[thm]{Proposition}
\newtheorem{cor}[thm]{Corollary}
\theoremstyle{remark}
\newtheorem*{rem}{Remark}

\theoremstyle{definition}
\newtheorem{defn}[thm]{Definition}

\title{Dissipative particle systems on expanders}
\author{
John Haslegrave\thanks{School of Mathematical Sciences, Lancaster University, UK.} \and 
Peter Keevash\thanks{
Mathematical Institute, University of Oxford, UK.}
}
\begin{document}
\maketitle
\begin{abstract}
We consider a general framework for multi-type interacting particle systems on graphs,
where particles move one at a time by random walk steps, 
different types may have different speeds,
and may interact, possibly randomly, when they meet.
We study the equilibrium time of the process, by which we mean
the number of steps taken until no further interactions can occur.
Under a rather general framework, 
we obtain high probability upper and lower bounds
on the equilibrium time that match up to a constant factor and are of order $n\log n$
if there are order $n$ vertices and particles.
We also obtain similar results for the balanced two-type annihilation model of chemical reactions;
here, the balanced case (equal density of types) does not fit into our general framework
and makes the analysis considerably more difficult.
Our models do not admit any exact solution as for integrable systems
or the duality approach available for some other particle systems,
so we develop a variety of combinatorial tools for comparing
processes in the absence of monotonicity.
\end{abstract}

\section{Introduction}

Interacting particle systems have been intensively studied since the 1970s
to model a variety of phenomena in Statistical Physics and Mathematical Biology,
such as spin systems, chemical reactions, population dynamics and the spread of infections.
A precise analysis of such systems in generality seems far out of reach,
although there is a large theory for some classical special models that are more tractable, 
due to admitting a variety of techniques based on monotonicity or duality
(see \cite{liggett1985interacting, liggett1999stochastic})
or even exact methods based on the rich theory 
of Integrable Systems, Random Matrices and KPZ Universality (see \cite{deift2014random}).

On the other hand, there are natural classical models that do not admit such techniques and
for which the theory is much less developed. An important example is the two-type annihilation model
for chemical reactions, studied classically on integer lattices $\mathbb{Z}^d$, 
where a celebrated series of papers by Bramson and Lebowitz \cite{BL90,BL91,BL01}
gives order of magnitude estimates for the site occupancy probability.
In this model, there are two types of particles and interactions only occur between
two particles of different type, which makes the analysis particularly difficult.

A single type model introduced by Erd\H{o}s and Ney \cite{EN74} is easier to analyse,
although even here it was an open problem whether the origin would be occupied infinitely often,
solved in one dimension by Lootgieter \cite{Loo77} and in higher dimensions by Arratia \cite{Arr81,Arr83}. 
There is also a substantial physics literature starting from \cite{EF85,KRL95} on annihilation with ballistic motion
which still has many open problems, such as the Bullet Problem (see e.g.~\cite{dygert2019bullet}).
The analysis of Bramson and Lebowitz is very model specific, 
relying on the lattice structure and both particle types moving at the same speed.
Less is known if we move away from these assumptions, although Cabezas, Rolla and Sidoravicius \cite{CRS18} 
showed even if the speeds differ that the system remains site recurrent on any infinite generously transitive graph.

Considering the real-world motivations for interacting particle systems, it is natural to consider systems
on finite graphs and ask for the asymptotics of key parameters when the number $n$ of vertices is large.
Here there is also a large literature (discussed in more detail below), 
dating at least to the early 80s, when Donnelly and Welsh \cite{DW83}
studied coalescing particles on finite graphs via duality with the voter model. 

\subsection{Balanced two-type annihilation}\label{sec:bal-intro}

Before introducing our general models, 
we will discuss the following two-type annihilation model for finite graphs,
recently considered by Cristali, Jiang, Junge, Kassem, Sivakoff and York \cite{2-type}.
Fix a graph with $n$ vertices and initialise with at most one particle at each vertex,
so that there are equal numbers of red and blue particles and at most one vertex is unoccupied;
we refer to this setting as \emph{balanced}. At each time step, 
with probability $p$ a red particle is chosen, uniformly at random from the remaining red particles,
or otherwise a uniformly random blue particle is chosen.
Without loss of generality $0\leq p\leq 1/2$, i.e.\ blue moves at least as fast as red.
The chosen particle performs a simple random walk step. 
If it reaches a vertex with one or more particles of the opposite colour, 
it mutually annihilates with one such particle. This process almost surely 
eventually terminates with no particles remaining.
The key quantity of interest is the time taken for this to happen,
which we call the \emph{extinction time}.

We will see below how the analysis of unbalanced annihilation 
can be subsumed in that of a much more general class of models (see Corollary \ref{cor:three-four-five}). 
However, the analysis of balanced annihilation is much harder  
due to the total number of particles changing dramatically over time.
Nevertheless, we are able to prove similar results for this model,
determining with high probability the extinction time up to constant factors,
for any initial particle distribution on any regular graph with sufficient spectral expansion
(for an introduction to expanders and their applications, see \cite{hoory2006expander}). 

\begin{thm}\label{thm:main}
Let $G$ be a regular graph on $n$ vertices with spectral gap at least $0.425$. 
Consider balanced two-type annihilation on $G$ 
from any starting configuration and let $T$ be the extinction time.
Then $cn\log n\leq T\leq Cn\log n$ with high probability and in expectation, 
where $c$ and $C$ are absolute constants.
\end{thm}

We make the following remarks on Theorem \ref{thm:main}:
\begin{enumerate}[noitemsep]
\item The upper bound on $T$ holds for any regular expander sequence 
(i.e.~regular graphs $G$ with spectral gap uniformly bounded away from zero).
\item The spectral gap condition holds with high probability for a random $d$-regular graph for any $d\geq 177$.
\item The high probability bound is sufficient to imply that $\mean{T} = \Theta(n\log n)$.
\end{enumerate}
The corresponding upper bound for unbalanced two-type annihilation, as one of a much more general class of models, is provided by Theorem \ref{thm:persistent}. However, this does not cover the case of balanced two-type annihilation, since the proof of Theorem \ref{thm:persistent} depends on the fact that in the unbalanced case (and more generally in the class of models covered), $\Theta(n)$ particles survive throughout the process.

The special case $p=0$ (stationary red particles) is equivalent to the particle-hole model 
of Cabezas, Rolla and Sidoravicius \cite{CRS14}, which they show 
has essentially the same behaviour as a system of activated random walks with infinite sleep rate. 
Here our lower bound applies for any regular graph (regardless of expansion)
assuming a uniformly random starting state.

\begin{thm}\label{thm:p=0}
Let $G$ be a regular graph on $n$ vertices.
Consider balanced two-type annihilation with stationary reds 
from a uniformly random starting configuration and let $T$ be the extinction time.
Then $\mean{T}\geq 0.08n\log n$ for $n$ sufficiently large, and with high probability $T>0.04n\log n$.
\end{thm}

The lower bound in Theorem \ref{thm:p=0} may fail starting from 
a configuration that is adversarial rather than random:
e.g.~if $G$ is a disjoint union of small even connected components
and the types are balanced in each component then $T$ is $\Theta(n)$
in expectation and with high probability. The same example shows that our other results below
for general regular graphs (Theorems \ref{thm:no-meeting} and \ref{thm:predator})
would fail if we considered an adversarial rather than random configuration.

While Theorem \ref{thm:main} considers the worst-case in the setting of \cite{2-type},
it is much easier to analyse the `absolute worst-case', permitting multiple particles at the same vertex 
and maximising over any speeds and starting positions. We obtain the following bounds,
which are optimal up to a small multiplicative constant and valid for any graph $G$.
The key parameter is the `worst-case hitting time' $H_{\max}(G) = \max_{x,y\in V(G)}H_x(y)$,
where the `hitting time' $H_x(y)$ is the expected number of steps taken by a random walk
on $G$ started at $x$ and stopped when it first hits $y$. 

\begin{prop} \label{prop:worst}
Fix any graph $G$. Let $T_k$ be the worst-case expected extinction time 
for two-type annihilation starting from $k$ particles of each type, 
maximised over any speeds and starting positions (which may coincide). 
Then $k H_{\max}(G)\leq T_k<4k H_{\max}(G)$.
\end{prop}

\subsection{Dissipative particle systems} \label{subsec:dissipative}

Now we will consider a broad class of models for processes with various types of particles 
moving diffusively and interacting in a possibly random way that is a function of their types. 
Our models will be \emph{dissipative}, meaning that each type of particle has an associated \emph{energy} 
(some positive real number) and that the total energy of all particles is strictly decreased
by any \emph{effective reaction} (\ie reaction where the output differs from the input);
for example, annihilation models are dissipative (for any energies). This condition gives
quite a different flavour from those models studying population dynamics with reinforcement
(see the survey \cite{Pem07} and \cite{AF22} for a recent result combining annihilation and growth).
It is also perhaps similar in spirit to the Dissipative Particle Dynamics (DPD) models
used for practical simulations of hydrodynamic phenomena (see \cite{EW17}).

Our models are specified by giving for each `input' pair $(A,B)$ of types 
a probability distribution $\pi_{AB}$ on `output' sets $S$ of types
(here $\pi_{AB}$ is fixed, independently of $n$).
If a particle of type $A$ arrives at a vertex and `meets' a particle of type $B$
then the process removes both particles and adds a new particle 
of each type in $S \sim \pi_{AB}$ independently of all other randomness.
(One could distinguish the ordered pairs $(A,B)$ and $(B,A)$,
but for simplicity we will assume $\pi_{AB}=\pi_{BA}$.)
For example, for two-type annihilation  these `distributions' 
are supported on a single output for each input:
the output for $\{A,B\}$ is $\varnothing$, for $(A,A)$ is $\{A,A\}$ and for $(B,B)$ is $\{B,B\}$. 
When a particle arrives at a vertex with one or more other particles already present, 
it is considered to \textit{meet} each such particle in a random order
until either an effective reaction occurs or all reactions are ineffective.
(This seems a natural model for such multiple collisions, but one could consider other options.)

We may reach an imbalance of types even starting from a balanced position,
and it is anyway often natural to start from a position of imbalance 
(e.g.~as considered for annihilating particles by Bramson and Lebowitz \cite{BL90}).
To allow for imbalance in our model, we assign speeds to types, 
such that the sum over types of their speeds is $1$,
and in each round of the process we select the particle 
to move with probability proportional to its speed.
For example, for two-type annihilation as above
we assign speeds $p$ to `red' and $1-p$ to `blue'.
Typically, we think of fixing the initial number of each type of particle in advance,
so that there are $n$ particles in total, and of the distribution of these particles
as chosen either randomly or adversarially. For our general models 
we consider the speeds to be fixed independently of $n$,
although we note that in our above results on two-type annihilation
we allow $p$ to depend on $n$.

Our key quantity of interest for these models is the time taken until no further reaction is possible,
which we call the \emph{equilibrium time}. For balanced two-type annihilation this is the extinction time,
for which we obtain $\Theta(n \log n)$ bounds, as stated above. 
We will prove the same bounds on the equilibrium time of general dissipative models, 
under an additional assumption regarding `active agents', 
to be described using the following definitions.

We say that a type is \textit{persistent}, for a given set of reactions and initial densities of the types, 
if there is some $\eps>0$ such that the density of that type must be at least $\eps$ 
throughout any possible sequence of reactions in a mean-field setting
(meaning that we allow any two particles to meet at each step, ignoring the graph $G$).
Similarly, we say that the particle system is \textit{persistent} if there is some $\eps>0$ such that 
the total density of particles must remain at least $\eps$ in a mean-field setting.
(This is a strictly weaker condition than having a persistent type.)
A type is \emph{ephemeral} if it is not persistent.
For example, in two-type annihilation, if the initial densities of red and blue differ
by some $\eps>0$ then the denser type is persistent and the other is ephemeral,
whereas in balanced two-type annihilation both types are ephemeral.

We say that the particle system is \emph{agential} 
if there are no effective reactions between two ephemeral types.
The interpretation of this terminology is that in some settings 
we think of particles of persistent types as `active agents'
and those of ephemeral types as `passive reagents'.
Note that by definition no reaction between persistent types is possible,
since if it were we could invoke it repeatedly until one type was eliminated, 
contradicting persistence. Thus in an agential system 
every reaction involves one persistent type and one ephemeral type.

Our general upper bound is as follows.

\begin{thm}\label{thm:persistent}
Let $G$ be a graph on $n$ vertices from a regular expander sequence.
Consider any dissipative agential particle system on $G$
starting from any configuration with at most $n$ particles.
Then the equilibrium time is $O(n\log n)$
with high probability and in expectation.
\end{thm}

The assumptions in Theorem \ref{thm:persistent} are all `somewhat necessary', for the following reasons.
\begin{enumerate}[noitemsep]
\item If we allowed the speeds of types to depend on $n$ then collisions involving two types of speeds $o(1)$
may be necessary for equilibrium but take too long to occur.
\item Dissipative dynamics rules out reversible models that never reach equilibrium, such as `catalysed modification' 
$A+c\to A+d$ and $B+d\to B+c$ (in our examples we adopt the convention that upper-case letters are persistent types).
\item If we allowed reactions between ephemeral types, one could introduce an inert persistent type 
to an otherwise ephemeral process, with such an overwhelming number of inert particles that reactions are rare.
Less trivially, consider a system such as $A+c\to b$ and $b+b\to c$
where $A$ is dense enough to be persistent. Then the penultimate reaction requires two specific particles of type $b$ to meet, 
and since there are still $\Theta(n)$ other particles remaining, this typically takes time $\Theta(n^2)$.
\end{enumerate}

For a matching lower bound, one should clearly also assume that $\Theta(n)$ reactions are necessary for equilibrium.
Equivalently, one should show that $\Omega(n\log n)$ steps are necessary to eliminate any type with positive initial density.
We believe that this holds for any such systems as above.
We support this belief by proving it when the initial positions of particles 
are independent stationary random vertices and all types have the same speed, 
for any regular graph (not necessarily an expander).
We deduce this from the following result of independent interest
on `lonely walkers' in a system of non-interacting random walks.

\begin{thm}\label{thm:no-meeting} (Lonely Walkers)
Let $G$ be a regular graph on $n$ vertices. 
Consider $n$ independent random walks, starting from stationarity, 
running for time $0.1n\log n$, with one walk moving at each time step. 
Then with high probability there are $\Omega(n^{3/4})$ walks 
that have never met any other walk.
\end{thm}

\begin{rem}
We give explicit constants in Theorem \ref{thm:no-meeting} for concreteness.
The same proof shows that for any $\eps>0$ there is $\delta>0$ 
so that if the walks run for time $\delta n\log n$ then  with high probability 
there are $\Omega(n^{1-\eps})$ walks that have never met any other walk.
\end{rem}

\begin{cor}\label{cor:survival}
Let $G$ be a regular graph on $n$ vertices. 
Consider any persistent particle system on $G$, 
where all types have the same speed,
and the starting locations of particles 
are independent and stationary.
Then with high probability, any type of positive initial density
has not yet been eliminated after $O(n\log n)$ steps.
\end{cor}

Our results apply to a rather broad class of models,
whereas much of the previous related literature considers
deterministic interactions for particles of one or two types;
here one can explicitly list the small number of possible models,
each of which has been studied in its own right.
We will discuss the implications of our results 
for these models in Section \ref{sec:one-two}.

\subsection{Simultaneous movement and the Big Bang}

Even for the most well-understood model of coalescence,
obtaining a precise understanding of the early evolution 
is widely considered a very difficult problem;
this has been dubbed the ``Big Bang regime'' by Durrett (see e.g.\ \cite{Dur10}).
A recent breakthrough on this problem for coalescing particles
on constant-degree random regular graphs
was achieved by Hermon, Li, Yao and Zhang \cite{big-bang};
comparable results were only previously known on the integer lattice \cite{BG80}, 
where transience plays a key role. For other models, such as annihilation,
comparable results are still unknown.

Much of the previous literature on interacting particles
considers models where particles move simultaneously,
rather than individually as in our models defined above.
Considering simultaneous movements side-steps the Big Bang question,
as then the Big Bang is so fast that it is insignificant compared to the equilibrium time,
which is dominated by the slower later stages.
By contrast, as our models have individual moments we cannot ignore the Big Bang:
we need to consider all particles, whereas in synchronous models 
only the longest-surviving particle matters. 
See also \cite{RSSS19} for a related model of internal diffusion-limited aggregation (IDLA),
where results are obtained for simultaneous movements, 
but the problem for individual movements remains open.
This distinction is only significant for ephemeral models, 
such as coalescence or balanced two-type annihilation,
as for persistent models the two settings are essentially equivalent:
the equilibrium times are related by a $\Theta(n)$ factor.

Consequently, while our methods are able to deal with both settings, 
we state our results for the more difficult case of individual movements.
Another minor advantage is that this lends itself more naturally 
to considering types of different speeds (although these could be implemented
in a simultaneous model by moving each particle independently with probability equal to its speed).
Although a speed differential is to be expected in almost all applications 
(e.g.~owing to different sizes of reacting molecules or enzymes and substrates), 
most previous work does not incorporate this factor, perhaps due to the focus on simultaneous movement. 

\subsection{Special models}\label{sec:one-two}

Now we will specialise to models with one or two types and a single deterministic reaction which does not increase the total number of particles. For the two-type case we also assume that the reaction occurs between the two types (which we call $A$ and $B$), since otherwise we would have a one-type system with additional inert particles. While it would be natural to permit multiple reactions, allowing a combination of one-type and two-type reactions, this would significantly increase the number of possibilities. Since we are not aware of previous work covering any such models, we do not consider them separately here.  

The possible `special models' up to relabelling $A \leftrightarrow B$ are as follows.

\begin{itemize}
	\item One type:
	\begin{enumerate}[(i),nosep]
		\item $A+A\to\varnothing$: one-type annihilation.
		\item $A+A\to A$: coalescence.
	\end{enumerate}
	\item Two type:
\begin{enumerate}[(i), resume*]
	\item $A+B\to\varnothing$: two-type annihilation.
	\item $A+B\to A$: predator-prey.
	\item $A+B\to A+A$: infection / communication.
\end{enumerate}
\end{itemize}
Besides the annihilation and coalescence models discussed above,
we also see models for (iv) predators of type A eating prey of type B,
and (v) infected individuals of type A infecting healthy individuals of type B.
The following, which includes the case of unbalanced two-type annihilation as mentioned in Section \ref{sec:bal-intro}, is immediate from Theorem \ref{thm:persistent} and Corollary \ref{cor:survival}.

\begin{cor}\label{cor:three-four-five}
Let $G$ be a graph on $n$ vertices from a regular expander sequence.
Consider any of the special models (i)--(v)
from any initial configuration with at most $n$ particles,
where each type has constant speed and positive initial density.
For model (iii), suppose also that the type densities 
differ by some $\eps>0$ independent of $n$.
Then the equilibrium time is $O(n\log n)$
with high probability and in expectation, and is $\Theta(n\log n)$ if the 
initial positions of particles are independently uniform and types have the same speed.
\end{cor}

Moreover, we have a stronger form of Corollary \ref{cor:three-four-five}:
the assumption that both types have positive density is not needed for the upper bound,
except that in models (iv) and (v) we need type A to have positive density.
We have not determined the precise conditions under which the conclusion holds,
but we observe that our conditions cannot be significantly relaxed. 
For example, if we have $k=\omega(1)$ particles of type $B$, 
the proof of Theorem \ref{thm:persistent}  gives an $O(n\log k)$ bound, 
so the lower bound cannot allow $k$ to be subpolynomial. 
The following result on the predator-prey model also
shows that the upper bound fails if there are $o(n/\log n)$ predators.

\begin{thm}\label{thm:predator}
Consider predator-prey dynamics on a regular graph $G$ on $n$ vertices, 
starting from $k$ predators and $n-k$ prey on distinct vertices,
where $4 \le k\leq n/\log n$ and both types have some nonzero speed independent of $n$.
Let $T$ be the equilibrium time. Then
	\begin{enumerate}[(a),nosep]
		\item If the starting configuration is selected uniformly at random then $\mean{T}=\Omega(n^2/k)$.
		\item If $G$ is from an expander sequence then $\mean{T}=O(n^2/k)$ for any starting configuration.
	\end{enumerate}
\end{thm}

\subsection{Comparison with previous results}\label{prev}

Cooper, Frieze and Radzik \cite{CFR09} gave a unified treatment 
of the special models considered above,
with three significant additional assumptions:
(a) the initial populations are polynomially small and randomly distributed,
(b) the underlying geometry is a random regular graph, and
(c) the types move simultaneously with equal speeds.
Removing assumption (a) is the main difficulty in our work
(see the above discussion of the Big Bang),
although considering general spectral expanders rather than (b)
and variable speeds rather than (c) also poses additional challenges.
 
Our model allowing for different speeds is much more general, 
as can be seen by considering the possibility of a zero speed (stationary) type,
which one might at first think would be simple to analyse,
but in fact is often an interesting and difficult model in its own right.
For example, balanced annihilation with one type stationary
strongly resembles the IDLA setting of \cite{RSSS19} mentioned above, 
except that in their model the moving particles start from what we might expect to be 
the worst-case scenario of all having accumulated at a single vertex. 

For the predator-prey or infection models, making type A (predators or infected individuals)
stationary is not so interesting, as the equilibrium time is then just a sum of hitting times
from the type B particles to the set of type A particles. 
In both cases, making $B$ particles stationary is more interesting. 
Predator-prey dynamics with $k$ moving predators and $n-k$ stationary prey 
is equivalent to the cover time of the graph by $k$ random walkers,
which was analysed in great detail by Rivera, Sauerwald and Sylvester \cite{RSS23}. 
Infection with stationary healthy individuals  is closely related to the frog model,
where we start from one active frog moving among a system of sleeping particles, 
with the latter becoming active frogs once their vertex is visited. 
(The frog model often has the further complications that active frogs can die
and there may be a random number of sleeping particles at each site.)
Much of the frog literature concerns propagation on infinite graphs,
although there are some results on the cover time (which is analogous to our model)
for special graphs by Hermon \cite{Her18} and Hoffman, Johnson and Junge \cite{HJJ19}.

Coalescence has been much studied as the dual process of the voter model,
where particles do not move but update their opinions to that of a random neighbour.
This was first considered on general graphs some 40 years ago by Donnelly and Welsh \cite{DW83}. 
More recently, Cooper, Els\"{a}sser, Ono and Radzik \cite{CEOR} gave a general upper bound 
in terms of the spectral gap and degree variability of the underlying graph. 
Haslegrave and Puljiz \cite{HP17} analysed a generalisation of the voter process 
which permits two opinions of different persuasiveness, similar to our variability in speed;
one of their results is that the complete graph has the smallest expected equilibrium time among regular graphs.

Finally, the balanced two-type annihilation model that we consider in great detail in this paper
was previously only analysed by Cristali et al.~\cite{2-type} for the complete graphs (the mean-field case) and stars.
For these graphs, the one-type annihilation model is easy to analyse, but the two-type process is much harder, 
owing to the possibility of multiple particles of the same type occupying a single site. 
While they were able to give precise results on the extinction time for stars, for the mean-field case 
they gave a lower bound of $2n\log n$ and an upper bound of $20n\log^2n/\log\log n$, 
thus leaving the open problem of showing that it is $\Theta(n\log n)$;
as discussed above, we prove this on all regular graphs with sufficiently strong spectral expansion.
Moreover, in \cite{clique} we determined the mean-field extinction time asymptotically as $(1+o(1))n\log n$,
for any speeds of the particles, even if the ratio of speeds is allowed to grow with $n$,
which creates additional difficulties (bounded speed ratio is assumed in~\cite{2-type}).
Our methods in \cite{clique} are specific to the complete graph 
and have no overlap with those introduced in this paper. In particular, for the complete graph we use the fact that
the probability of a collision at each step depends only on the number of vertices occupied by each type of particle, and 
that if the number of sites occupied is sufficiently small compared to the total number of particles, there is a strong
tendency for this disparity to self-correct, irrespective of the precise configuration of particles; neither of these facts apply to
general graphs.

\subsection{Methods and organisation}

As indicated above, the generality of our models requires the development of new combinatorial comparison techniques.
A basic strategy in many of our proofs is to consider a short segment of the process, at the start of which we couple the
`true' particles with `fake' particles, where the fake particles go on to follow non-interacting walks, and so gradually
become decoupled from the true particles as time proceeds. We aim to show that with high probability 
a suitable set of collisions have occurred in the non-interacting process, 
and deduce that a certain number of reactions must have taken place. 
The time period of the coupling must be chosen so that  the total number of particles 
does not change by more than a constant factor,
as otherwise the rates of passing of true time and fake time are not comparable.
This highlights the difficulty of non-agential models such as balanced two-type annihilation, 
where the total number of particles decays and many time periods are necessary.

We will divide the paper into two parts, where in the first part we treat general agential dissipative models,
and in the second part we consider the harder setting of balanced two-type annihilation.
The first section of the first part introduces a Moving Target Lemma that will be useful throughout the paper.
An intuitive statement is that we get a roughly tight estimate for the probability of an approximately 
stationary random walk on an expander hitting a target, where the target is allowed to move,
but we `wait' a constant number of steps between each `attempt' to hit the target
to allow the conditional distribution of the walker to settle back towards approximate stationarity.
The second section of the first part introduces a poissonisation technique for handling dependencies
that will also be  useful throughout the paper; we combine this with a second moment argument to prove
our result on lonely walkers, which implies the lower bounds for our general models.
The remainder of the first part treats some special models, using a variety of ideas
on hitting times of random walks and toppling of  `abelian sandpiles'.

In the second part on balanced two-type annihilation, the techniques developed in the first part
are combined with additional technical arguments needed to address the new difficulties 
that arise in the proofs of both the upper and lower bounds. For the upper bound,
the Moving Target Lemma can still be used for comparison, although the previous basic argument
does not suffice, so is replaced by a refined argument using Hall's Matching Theorem.
Furthermore, there are various regimes for the number of remaining particles, where in some
regimes the approximate stationarity needed for the Moving Target Lemma can only be guaranteed
for the faster blue particles, so further technical arguments are needed to either control the
distribution of the slower red particles, or show that even a moderately bad distribution of red particles
can be handled by tighter concentration inequalities. 

For the lower bound, we start by illustrating the 
arguments for the one-type annihilation model, which has many of the key ideas but the important
simplifications that all particles move at the same speeds and each vertex has at most one particle.
The key step is to show that when $k=n^x$ particles remain with $x<1$ then 
 there are constants $c_1,c_2,c_3>0$ so that
starting from any configuration $A$ of $k$ particles, 
with probability at least $c_1$ we need at least $c_2 n$ steps
to reduce the number of particles to $c_3 k$.
We prove this via results relating return times to the spectral gap 
and using trajectory reversal arguments to control meetings by returns.
For the two-type model, there are additional arguments to handle variation in speeds,
and a subtle application of Reimer's inequality on disjointly occurring events,
which is used to rule out `catastrophic collapse', \ie  reducing from $n^c$ to $n^{c-o(1)}$ 
remaining particles so rapidly that the surviving particles do not have time to mix.

\subsection{Notation}\label{sec:notation}

For a $d$-regular connected graph $G$, let $\bm Q$ be the Laplacian, 
let $\bm P$ be the transition matrix of the lazy random walk and let $\bm A$ be the adjacency matrix. 
Then we have $\bm Q=d\bm I-\bm A$ and $\bm P=\frac{1}{2}\bm I+\frac{1}{2d}\bm A$, 
giving $\bm Q=2d(\bm I-\bm P)$. Let $0=\lambda_1\leq\lambda_2\leq\cdots$ 
be the eigenvalues of $\bm Q$, and $1=\mu_1\geq\mu_2\geq\cdots \ge 0$ be the eigenvalues of $\bm P$. 
Then $1-\mu_2=\frac{1}{2d}\lambda_2$. We say that a sequence of graphs $G$ form an 
\emph{expander sequence} if $1-\mu_2(G)\geq c$ for some constant $c>0$, \ie $\lambda_2=\Theta(d)$.

We often consider the modification of our processes where we replace random walk 
steps by lazy random walk steps.
By suppressing the lazy steps in which a particle does not move we obtain a 
copy of the original process.
By Chernoff bounds, with high probability the number of lazy steps and normal steps 
are asymptotically equal,
so this modification essentially doubles the equilibrium time. We refer to a particle "moving" if it is is selected to take a step of the lazy random walk, even if that step has no effect.

Throughout the paper, our main focus is on systems of discrete-time random walks where exactly one 
particle is selected to move at each discrete time. We refer to a single step of this process as a ``time step''. For the purposes of the proofs, we vary this in two ways. First, and most commonly, we consider two linked copies of this process, where in each time step either one particle moves in each copy (in which case the two moving particles correspond), or one particle moves in one copy and nothing happens in the other. We sometimes refer to this as ``fake time'', with ``real time'' being the number of movements that have occurred in the primary copy. We ensure that a movement occurring in both copies has constant probability, so that ``real time'' and ``fake time'' differ only by a constant factor with high probability.

Secondly, we occasionally (see Lemmas \ref{lem:meeting-vs-return} and \ref{lem:mr2} and the proof of Theorem \ref{thm:no-meeting}) approximate by a continuous-time process where particles independently take (lazy) random walk steps at given rates (in order to remove dependencies), and then discretise by considering the positions of particles at integer times. For clarity, we refer to these discrete transitions, during which more than one particle might move, and an individual particle might move more than once, as ``time intervals''. Each time interval then corresponds to as many time steps as movements that occur in that interval.

We use standard asymptotic notation throughout and suppress notation
for rounding to integers where it does not affect the argument. 

\newpage

\part{General models}

In this part we prove our results on general agential dissipative models.
We start in Section \ref{sec:moving} by proving our general upper bound (Theorem \ref{thm:persistent}),
via the Moving Target Lemma (Lemma \ref{lem:hitting}) that will prove to be a useful tool throughout the paper.
In Section \ref{sec:lonely} we introduce the Poissonised model that will also be frequently useful,
and apply it to prove our result on lonely walkers (Theorem \ref{thm:no-meeting}),
which easily implies Corollary \ref{cor:survival}, and so the lower bounds in Corollary \ref{cor:three-four-five}.
The remainder of the part treats some special models.
In Section \ref{sec:topple} we apply various results on hitting times
and toppling of  `abelian sandpiles' to analyse 
the worst-case expected extinction time (Proposition \ref{prop:worst}) 
and the particle-hole model (Theorem \ref{thm:p=0} on stationary reds). 
We conclude the part in Section \ref{sec:predator}  by proving Theorem \ref{thm:predator}
on the predator-prey model with few predators.

\section{Hitting moving targets}\label{sec:moving}

In this section we prove our main upper bound for general models (Theorem \ref{thm:persistent}).
The key lemma is the following result on an approximately stationary random walk hitting a moving target.
Intuitively, such a walk should hit a target $k$-set with probability $O(k/n)$, so the probability of missing
the target $sn/k$ times should decay exponentially in $s$. However, the hitting events are not independent,
so the idea of the proof is to show that conditioning on missing a target only increases the $L^2$ distance
to stationarity by a constant factor, which is counterbalanced by a constant number of random walk steps
before the next target is considered.

\begin{lem}\label{lem:hitting} (Moving Target Lemma)
Let $G$ be a regular graph on $n$ vertices and 
$k,r,s$ be positive integers with $k \le n/8$ and $\mu_2(G)^{2r} < 1/17$.
Suppose $A_0, A_1,\ldots$ are deterministic sets of $k$ vertices.
Let a particle $p$ follow a lazy random walk $v_0,v_1,\ldots$ on $G$, where
\begin{equation}\sum_{v\in V(G)}(\prob{v_0=v}-1/n)^2\leq \frac{k}{4n^2}.\label{init-dist}\end{equation}
For each $i$ let $E_i$ be the event $v_{ir}\in A_{ir}$
and let $X_i$ be the event that $E_j$ does not hold for any $j<i$.
Let $\ell := \ceil{6n/k}$.
Then $\ee^{-18s} \le \prob{X_{s\ell}} \le \ee^{-3s}$.
\end{lem}

\begin{rem}
When applying Lemma \ref{lem:hitting}, we will be interested in whether the particle $p$
hits some other set of particles, located at the target set $A_i$ at step $i$.
This can be reduced to the setting of the lemma by revealing the times at which each particle moves
and all positions and movements of all particles other than $p$. The only randomness is then in the random walk
followed by $p$, and the corresponding target sets $A_i$ are deterministic conditional on the revealed randomness.
\end{rem}

\begin{proof}
We will prove the following claim for each $i$:
\begin{equation}\frac{k}{2n}\leq\prob{E_i\mid X_i}\leq \frac{3k}{2n}.\label{Ei-bounds}\end{equation}
This will suffice to prove the lemma, 
using $\prob{X_{s\ell}} = \prod_{i=0}^{s\ell} \prob{E_i^c \mid X_i}$
and $e^{-2t} \le 1-t \le e^{-t}$ for $t \in [0,1/2]$. 
To prove the claim, it suffices to show that
\begin{equation}\sum_{v\in V(G)}(\prob{v_{ir}=v\mid X_i}-1/n)^2\leq \frac{k}{4n^2}.
\label{small-dist}\end{equation}
Indeed, if  \eqref{Ei-bounds} failed then the triangle inequality would give
\[\sum_{v\in A_{ir}}\biggl|\prob{v_{ir}=v\mid X_i}-\frac{1}{n}\biggr|>\frac{k}{2n},\]
but this contradicts \eqref{small-dist}  by Cauchy--Schwarz.

It remains to show \eqref{small-dist}. We use induction on $i$.
It holds for $i=0$ by assumption. Now suppose inductively that it holds for some $i \ge 0$.
For each $j=0,\ldots,r$, let 
	\[d_j=\sum_{v\in V(G)}(\prob{v_{ir+j}=v\mid X_{i+1}}-1/n)^2.\]
Writing $q_j$ for the vector of $\prob{v_{ir+j}=v\mid X_{i+1}}$ for $v\in V(G)$, we have $q_j=\bm Pq_{j-1}$, for $\bm P$ as in Section \ref{sec:notation}. Note that $\bm P$ has an orthogonal basis of eigenvectors $z_1,\ldots, z_n$, where the Perron--Frobenius eigenvector $z_1$ has all entries $1/n$. Since $d_j=\|q_j-z_1\|^2$, and $q_{j-1}\cdot z_1=\|z_1\|^2$, it follows that $d_j\leq \mu_2^2d_{j-1}$.
As $X_{i+1} \Rightarrow E_i^c  \Rightarrow v_{ir} \notin A_{ir}$, we have
\begin{align}
\prob{v_{ir}& =v\mid X_{i+1}}=\frac{\prob{v_{ir}=v\mid X_{i}}}{\prob{E_i^c\mid X_i}} 1_{v\not\in A_{ir}}, \nonumber \\
\label{d0} \text {so} \quad d_{0} & =\frac{k}{n^2}+\sum_{v\not\in A_{ir}}\left(\frac{\prob{v_{ir}=v\mid X_{i}}}{\prob{E_i^c\mid X_i}}-\frac{1}{n}\right)^2.
\end{align}
We partition $V(G)\setminus A_{ir}$ as $(U^-,U^+)$, 
where $U^- = \Big\{ v: \frac{\prob{v_{ir}=v\mid X_i}}{\prob{E_i^c\mid X_i}}<1/n \Big\}$. 
Now
	\begin{align}\sum_{v\in U^-}\left(\frac{\prob{v_{ir}=v\mid X_{i}}}{\prob{E_i^c\mid X_i}}-\frac{1}{n}\right)^2
	&\leq\sum_{v\in U^-}\left(\prob{v_{ir}=v\mid X_{i}}  -\frac{1}{n}\right)^2 \leq \frac{k}{4n^2}\label{Uminus}\end{align}
by the inductive hypothesis \eqref{small-dist}.
We also have $\prob{E_j^c\mid X}^{-1}\leq (1-3k/2n)^{-1}\leq 1+2k/n$, 
by  \eqref{Ei-bounds}, which follows from the inductive hypothesis \eqref{small-dist}.
For any $v\in U^+$, we consider
\[ (ap_v-1/n)^2 - a^2(p_v-1/n)^2 = \frac{a-1}{n} \left(2ap_v - \frac{a+1}{n}\right) < \frac{2k}{n^2} \cdot \frac{5p_v}{4} \]
with $a = \prob{E_j^c\mid X}^{-1} \le 1+2k/n$ and $p_v = \prob{v_{ir}=v\mid X_{i}}$. Summing over $v$,
using $\sum p_v = 1$ and $\sum_v (p_v-1/n)^2 \le \frac{k}{4n^2}$ by  \eqref{small-dist}, we deduce
	\begin{equation}
		\sum_{v\in U^+}\left(\frac{\prob{v_{ir}=v\mid X_i}}{\prob{E_i^c\mid X_i}}-\frac1n\right)^2
		< \left(1+\frac{2k}{n}\right)^2\frac{k}{4n^2}+\frac{5k}{2n^2}<\frac{3k}{n^2}. \label{Uplus}
	\end{equation}
Combining \eqref{d0}, \eqref{Uminus} and \eqref{Uplus} gives $d_0\leq\frac{17k}{4n^2}$.
As $\mu_2(G)^{2r}<1/17$, we deduce $d_r\leq \frac{k}{4n^2}$, \ie \eqref{small-dist} holds for $i+1$, as required.
\end{proof}

Our next lemma provides a useful property of dissipative agential models.
We require the following definition. Let $\prec$ be a total ordering on the ephemeral types.
We say $\prec$ is an \emph{ephemeral ordering} if for any ephemeral type $x$,  
no particle can become type $x$ after all ephemeral types prior  to $x$ in $\prec$ are eliminated.

\begin{lem}\label{lem:ordering}
Any dissipative agential particle system has an ephemeral ordering.
\end{lem}

\begin{proof}
Consider any dissipative agential particle system. Recall that for such systems 
every reaction involves one persistent type and one ephemeral type.
We define a relation $\prec$ on ephemeral types as follows. 
If $x$ and $y$ are ephemeral types then $x\prec y$ if there is some 
effective reaction with input including $x$ and output including $y$.
We will show that $\prec$ can be extended to a total ordering,
which will then be the required ephemeral ordering,
as after all types prior to $x$ have been eliminated 
then all output ephemeral particles must come after $x$.

First we note that $\prec$ is irreflexive, \ie we cannot have $x\prec x$.
Indeed, we would have an effective reaction with input $\{A,x\}$ for some persistent type $A$
and output including $x$ but not $A$ (as the system is dissipative).
However, repeating this reaction can eliminate $A$, contradicting persistence.
It remains to show that $\prec$ has no directed cycles,
\ie we cannot find ephemeral types $x_1,\ldots,x_r$ 
such that $x_1\prec\cdots\prec x_r\prec x_1$. 
Suppose that such a cycle exists, where for each $i$ there is an effective reaction 
with input $\{A_i,x_i\}$ for some persistent type $A_i$
and output including $x_{i+1}$ (where $x_{r+1}:=x_1$).
From any starting configuration with at least one particle of type $x_1$,
we can invoke these reactions in order (all $A_i$ are available by persistence).
After these reactions, the total energy of ephemeral particles has not decreased,
whereas the total energy has decreased (as the system is dissipative),
so the total energy of persistent particles has decreased.
However, repeating this cycle will then decrease the energy of persistent particles
until some persistent type is eliminated, which is a contradiction.	
\end{proof}

Before proving the main result of the section, we pause to collect some
standard facts on mixing of random walks on expanders.
Recall that $\bm P=\frac{1}{2}\bm I+\frac{1}{2d}\bm A$
is the transition matrix of the lazy random walk on $G$,
which is a $d$-regular connected graph on $n$ vertices
with eigenvalues $1=\mu_1\geq\mu_2\geq\cdots \ge 0$.
Thus $P^t_{uv}$ for any vertices $u,v$ is the probability that 
the lazy random walk started at $u$ arrives at $v$ at time $t$. 
By a simple calculation (see e.g.~\cite[(3.1)]{Lov96}, noting that the stationary distribution is uniform),
\begin{equation}\label{converge}\lvert P^t_{uv}-1/n \rvert\leq \mu_2^t.\end{equation}
Thus the mixing time is $O(\log n)$. To be precise, for $t\geq \frac{3\log n}{-\log \mu_2}$ we have
\begin{equation}\label{prob-bounds}\frac1n-\frac1{n^3}\leq P^t_{uv} \leq
 \frac1n+\frac1{n^3}.\end{equation}

We conclude this section by proving our main upper bound on general models.

\begin{proof}[Proof of Theorem \ref{thm:persistent}]
Let $G$ be a graph on $n$ vertices from a regular expander sequence.
Consider any dissipative agential particle system on $G$
starting from any configuration with at most $n$ particles.
As no reaction between persistent types is possible,
to bound the equilibrium time it suffices to prove that 
in time $O(n\log n)$ all ephemeral types are completely eliminated
(with high probability and in expectation).
As noted above, we can consider the modified process with lazy random walks.

We consider the ephemeral types according to the ephemeral ordering provided by Lemma \ref{lem:ordering}. 
It suffices to show for each type $x$ in the ordering that, starting from any configuration where all types
prior to $x$ have been eliminated, with high probability and in expectation in time $O(n\log n)$,
either (a) all particles of type $x$ react effectively, 
or (b) there are at least $cn$ effective reactions, for some fixed $c>0$.
This indeed suffices, as option (a) can only occur once for each ephemeral type
and option (b) can only occur $O(1)$ times for a dissipative system.
We fix some type $A$ with some positive probability 
of an effective reaction when meeting $x$; note that $A$ must be persistent.

We first consider a `mixing phase', in which we run the process 
for time $T_1 = K_1 n\log n$, for some large constant $K_1$.
For the analysis, we consider a parallel system of non-interacting `fake' particles, 
starting in the same positions as the current particles, and maintain a partial pairing
between the true particles and the fake particles (initially a complete pairing with each true particle paired to the fake particle starting at the same vertex).
Paired particles become unpaired if the true particle in the pair participates in an effective reaction.
In each step, we select a particle to take a random walk step 
with probability proportional to its speed (as described in Section \ref{subsec:dissipative}),
where we simultaneously consider true and fake particles, 
but count paired particles as a single particle.
If a pair is selected then they take the same random walk step in both systems.
If an unpaired particle from one system is selected then the other system ignores this step.
We note that this coupling gives the correct marginal distribution for the true system.

At any step, by persistence of $A$, for some fixed $\alpha>0$
we can assume there are at least $\alpha n$ paired type $A$ particles
(otherwise there are $\alpha n$ effective reactions, so option (b) holds).
There is some constant $C$ depending on the model 
such that there can be at most $Cn$ particles by the dissipative property.
We say that a paired type $A$ particle is \emph{visible} if at least $\alpha/2C$ proportion
of the particles at its site (including itself) are paired type $A$ particles.
Then there are at least $\alpha n/2$ visible paired type $A$ particles.

Throughout, the true and fake populations both have $\Theta(n)$ particles (by persistence),
so both take $\Theta(n\log n)$ steps during the mixing phase.
(Using the assumption that the speeds are fixed independently of $n$;
henceforth this will not be explicitly mentioned.)
Also, by any time $t = \Omega(n\log n)$, 
with high probability each fake particle takes $\Theta(\log n)$ steps.
Conditional on the initial locations of the fake particles, 
and the number of random walk steps each has taken, all being $\Theta(\log n)$,
their locations are independent random variables
and approximately uniform in the sense of \eqref{prob-bounds}. 

Next we let $\eps = z\alpha/3$, with $z>0$ being
a lower bound on the probability that a type $x$ particle
will have an effective reaction with a visible paired type $A$ particle
when it arrives at a site containing such a particle.
We claim that for some $\eta>0$,
with high probability at each time $t$ with $\Omega(n\log n) = t = O(n^2)$,
any set of $\eps n$ paired type $A$ particles occupies at least $\eta n$ vertices.
To see this, first note by \eqref{prob-bounds} each paired particle is 
at any given site with probability at most $2/n$, 
independently given the above conditioning.
Let $E_t$ be the event that there is some set of $\eta n$ vertices
that includes the locations at time $t$ of at least $\eps n$ particles.  Then
	\[\prob{E_t}\leq\binom{n}{\eps n}\binom{n}{\eta n}(2\eta)^{\eps n}
	\leq (\exp(1+\eta/\eps)\eps^{-1}\eta^{1-\eta/\eps})^{\eps n} < \eta^{\eps n/2}\]
for $\eta$ sufficiently small. The claim follows by a union bound over $t$. 

Now we will consider the elimination of type $x$, via the Moving Target Lemma.
We will show that this occurs with high probability and in expectation
by time $T_2 = K_2 n\log n$, for some large constant $K_2$.
Fix any particle $i$ of type $x$. We consider the particle system up to the stopping time $\tau$
which is the minimum of $T_2$ and the time at which $i$ undergoes an effective reaction.
Thus $i$ has not yet undergone an effective reaction before time $\tau$.
We need to bound $\prob{\tau=T_2}$. We reveal (a) which particle moves at each time step
and (b) all movements and reactions not involving $i$; thus the only remaining randomness is in
the random walk and reactions of $i$.  We associate to every visible paired type $A$ particle $j$
independent Bernoulli random variables $Z_j$ with probability $z>0$,
which we couple to the process so that if $i$ meets $j$ and this has not previously occurred 
and $Z_j=1$ then an effective reaction occurs.
By Chernoff bounds, with high probability at all times $t$ with $\Omega(n\log n) = t = O(n^2)$
there is some set of $\eps n$ paired type $A$ particles $j$ each having $Z_j = 1$.
By the previous claim, we can assume that these occupy at least $\eta n$ vertices.

Now we apply Lemma \ref{lem:hitting}, with $k = \eta n$ (we can assume $\eta \le 1/8$)
and $s = 2\log n$, where $p=i$ and the sets $A_i$ are the vertices occupied
by the paired type $A$ particles $j$ with $Z_j=1$, as described above.
We conclude that with probability $1 - O(n^{-6})$ 
within $12\eta^{-1}\log n$ steps taken by particle $i$
it arrives at a vertex occupied by a particle with which it has an effective reaction.
Taking a union bound over $i$, with high probability every particle of type $x$
reacts in time $\Theta(n \log n)$. Repeated trials give the same bound in expectation.
As discussed above, we accumulate $O(1)$ such time periods, so this completes the proof.
\end{proof}

\section{Poissonised lonely walkers} \label{sec:lonely}

In this section we prove our result on lonely walkers (Theorem \ref{thm:no-meeting}),
which easily implies Corollary \ref{cor:survival}. 
We have previously noted that Corollary \ref{cor:three-four-five}
is immediate from Theorem \ref{thm:persistent} and Corollary \ref{cor:survival}.
Next we show how  Corollary \ref{cor:survival} follows from Theorem \ref{thm:no-meeting}.

\begin{proof}[Proof of Corollary \ref{cor:survival} assuming Theorem \ref{thm:no-meeting}]
Let $G$ be a regular graph on $n$ vertices. 
Consider any persistent particle system on $G$, 
where all types have the same speed,
and the starting locations of particles are independent and stationary.
We couple this to a system of non-interacting walks corresponding to
fake particles as in the proof of Theorem \ref{thm:persistent},
where the starting locations are independent and stationary
and we randomly assign types according to the starting distribution. 
We consider the coupled process up to time $0.1n\log n$, noting that 
by persistence both systems have $\Theta(n)$ particles at all times,
so the true system takes  $\Omega(n\log n)$ steps with high probability.
By Theorem \ref{thm:no-meeting}, with high probability there are $\omega(1)$ walks that never met another walk.
For any given type $A$ of initial density $c>0$, each such walk has probability $c$
of being assigned type $A$, so with high probability some such walk
corresponds to a surviving type $A$ particle.
\end{proof}

It remains to prove the lower bound on lonely walkers.
The key idea is to remove dependencies via the following Poisson approximation.

\emph{Poissonised model}: We start $\Po(1.1)$ particles at each vertex independently. Each independently takes random walk steps at rate $1/n$, for time $N:=0.11n\log n$. However, we then discretise by considering the transitions between integer time points. Thus we divide into $N$ time intervals, and retain only the information of what movements occur in an interval. Note that in each interval, each particle independently takes $\Po(1/n)$ random walk steps.

We note that the initial distribution of particles is stationary under these dynamics:
at any time there are $\Po(1.1)$ particles at each vertex independently.
This follows from two well-known properties of Poisson variables,
used henceforth without further comment:
\begin{enumerate}[nosep]
\item (Combining) If $X,Y$ are independent Poissons then $X+Y$ is Poisson.
\item (Splitting) If $X$ is Poisson and $Y \sim \operatorname{Bin}(X,p)$ then $Y$ is Poisson.
\end{enumerate}
Before starting the proof, we record a tail bound for Poisson variables from \cite{Poisson}: 
	\begin{equation}\label{poisson-bounds}
	\prob{\operatorname{Po}(\lambda)\geq r}\leq \exp(-r\log(r/\lambda)+r-\lambda) \quad \forall r>\lambda.
	\end{equation}

\begin{proof}[Proof of Theorem \ref{thm:no-meeting}]
Let $G$ be a regular graph on $n$ vertices. 
Consider $n$ independent random walks, starting from stationarity, 
running for $N':=0.1n\log n$ discrete time steps, with one walk moving at each time step. 

We attempt to couple this system to a subset of the particles of the Poissonised model 
described above, as follows. First, if the Poissonised model has at least $n$ particles,
select $n$ of these particles uniformly at random, and independently of their movement. Then, for each time interval, order the
movements occurring in that interval uniformly at random. Now divide into time steps such that exactly one selected particle moves in each time step. Provided the $n$ selected particles move at least $N'$ times in the $N$ time intervals, we can couple the particles of the original process to the selected particles from the Poissonised model so that they follow the same movements for the first $N'$ time steps of the latter. This coupling succeeds provided it is possible to select $n$ particles, and these made at least $N'$ movements in $N$ time intervals, both of which happen with high probability.

It suffices to show that the Poissonised model
with high probability has $\Omega(n^{3/4})$ particles that have never met another particle,
as then with high probability $\Omega(n^{3/4})$ of these were selected, and correspond
to walks in the original system that have never met any other walk.

For each particle and time interval of the Poissonised model, we define the vertices visited by that particle in that interval to consist of the vertex occupied by that particle at the start of that interval together every vertex that is the end of a step taken by that particle during that interval, if any.

We say that two particles of the Poissonised model \textit{collide strongly} 
if they are at the same vertex at some integer time (\ie at the start or end of some time interval). 
We say that they \textit{collide weakly} if they do not collide strongly, 
but there is some vertex visited by both particles in the same time interval. 
For two particles to collide weakly, one of the following must occur in some time interval:
	\begin{itemize}[nosep]
		\item both particles move more than once;
		\item one particle moves more than once, visiting a vertex occupied by the other particle at the start or end of the interval;
		\item both particles move exactly once, with the departure vertex of one being the arrival vertex of the other.
	\end{itemize}

We claim that with high probability only $O(\log n)$ particles have any weak collisions.
To see this, first note that $\prob{\Po(1/n)\geq 2}=\Theta(1/n^2)$ and $\prob{\Po(1/n)\geq 3}=\Theta(1/n^3)$. 
Thus with high probability no particle ever moves more than twice in a single time interval, 
and there are only $O(\log n)$ pairs $(a,t)$ such that particle $a$ moves twice at time $t$, 
with all such pairs being disjoint (\ie no repeated particles or times).
We assume this is the case and reveal the pairs $(v,t)$ such that 
some particle visits vertex $v$ while moving twice in the interval $[t-1,t]$.
By independence, the total number of other particles at $v$ either at time $t$ or time $t-1$,
over all such pairs $(v,t)$, is Poisson with mean $O(\log n)$, 
so with high probability is $O(\log n)$ by \eqref{poisson-bounds}.
Similarly, given that some particle moves from $u$ to $v$ in a given interval, the number of other particles moving to $u$ or from $v$ in the same time interval, in total over all such moves, is with high probability $O(\log n)$.
The claim follows.

For each vertex $v$, let $A_v$ be the event that 
exactly one particle starts at $v$ and it has no strong collisions.
Let $X$ be the number of these events that occur. 
It suffices to show with high probability $X = \Omega(n^{3/4})$.
We will show that $\mean{X}=\Omega(n^{3/4})$ 
and $\Var(X)=\widetilde{O}(\mean{X})=o(\mean{X}^2)$, 
which will suffice by Chebyshev's inequality.

We define a \textit{trajectory} to be a sequence $T=(v^{(T)}_0,\ldots,v^{(T)}_N)$ of vertices.
We say that a particle \emph{follows} trajectory $T$ if for each integer time $t$ its position at time $t$ is $v^{(T)}_t$. 
We say trajectories $T$ and $T'$ \emph{meet} at an integer time $t$ if particles 
following these trajectories occupy the same position at time $t$ but not at time $t-1$. 
Thus particles collide strongly if and only if their trajectories meet.

By construction of the Poissonised model
\begin{enumerate}[nosep, label=(\alph*)]
\item each particle independently has some fixed probability of following $T$,
\item the number $N_T$ of particles following $T$ is $\Po(w_T)$ for some real `weight' $w_T$, and
\item the variables $(N_T)_T$ are independent.
\end{enumerate}
We call a trajectory \emph{valid} if it has at most $\log n$ movements.
By \eqref{poisson-bounds}, the expected number of particles following invalid trajectories is $o(1)$,
so we will ignore such trajectories, as with high probability this does not affect the process.

Next we estimate $\mean{X} = \sum_v \prob{A_v}$.
Fix $v$ and consider the event that exactly one particle starts at $v$
and follows a trajectory $T$ that moves at most $0.11\log n$ times. 
This has probability $1.1\exp(-1.1)(1-o(1))\geq \exp(-1.1)$.
We estimate the  total weight of trajectories $T'$ meeting $T$ as follows.
For each time interval at which $T$ moves, the total weight of trajectories $T'$ that meet $T$ at the end of that interval is $1.1$, so summing over at most $0.11\log n$ such intervals gives weight at most $0.121\log n$.
For each time interval during which some $T'$ moves to meet $T$ at the end of that interval,
the total weight of such $T'$ is at most $1.1 (1-e^{-1/n}) < 1.1/n$.
Summing over at most $0.11n\log n$ such intervals gives weight at most $0.121\log n$.
The total weight of such $T'$ is then at most  $0.242\log n < \tfrac{1}{4} \log n$,
so with probability $\Omega(n^{-1/4})$ no such trajectory has a particle.
Summing over $v$ we obtain $\mean{X} = \Omega(n^{3/4})$.

It remains to bound  $\Var(X)$. 
For each $v$, we partition $A_v$ into events $B_{v,T}$ 
that exactly one particle starts at $v$, 
it follows trajectory $T$, and no other particle meets it.  Writing
\begin{align*}
& D_v  :=\sum_u(\prob{A_u\mid A_v}-\prob{A_u})
\quad \text{ and } \quad
D_{v,T} := \sum_u(\prob{A_u\mid B_{v,T}}-\prob{A_u}), \\
& \text{we have} \quad \Var(X)=\sum_v \prob{A_v} D_v 
\quad \text{ and } \quad
D_v = \sum_u \sum_T \prob{B_{v,T}\mid A_v} D_{v,T}. 
\end{align*} 
Thus $\Var(X)/\mean{X}$ is a convex combination of the $D_{v,T}$,
so it suffices to show that each $D_{v,T} = o(\mean{X})$.
In fact, we will show $D_{v,T}=O(\log^3 n)$.

To bound $D_{v,T}$, note that starting from the unconditional process, 
setting the number of particles following $T$ to $1$, setting the number 
following any trajectory $T'$ meeting $T$ to $0$, and leaving all other trajectories unchanged,
we obtain a process with the same law as  the process conditioned on $B_{v,T}$.
 
Now, in order for some $A_u$ to occur after these changes but not before, 
the unconditioned process must have an occupied trajectory $T'$ starting at $u$,
such that $T'$ meets some trajectory $T''$ occupied in the unconditioned process 
but not after the changes, due to $T''$ meeting $T$.
Thus $D_{v,T}$ is bounded by the sum over  $u$ of the probability that 
some such pair $(T',T'')$ has particles in both trajectories,
which is the sum over such pairs of $O(w_{T'}w_{T''})$.

To estimate this sum we introduce some further notation.
We call a sequence $S$ of positions $v^{(S)}_i,\ldots,v^{(S)}_j$ with $0<i<j$
a \textit{subtrajectory} and say that a trajectory $T$ is \emph{consistent} with $S$
if $v^{(T)}_t=v^{(S)}_t$ for all $i\leq t\leq j$. 
We write $w_S$ for the total weight of all trajectories consistent with $S$.
Also, we write $w^*_S = w_S (1-e^{-1/n}) \sim w_S/n$ 
for the total weight of all trajectories consistent with $S$ that move in the time interval $[i-1,i]$.
	
We consider various cases for pairs $(T',T'')$ as above.
Suppose first that (a) $T''$ meets $T'$ before it meets $T$,
and (b) $T'$ moves in the time interval immediately before it first meets $T''$.
Let $t$ be the time when $T'$ and $T''$ first meet, let $S'$ be the 
subtrajectory of $T'$ from its start at $u$ up to time $t$,
and let $S''$ be the subtrajectory of $T''$ from $t$ up
to its first meeting with $T$. We can bound the sum over such $(T',T'')$
of $w_{T'}w_{T''}$ by the sum over such $(S',S'')$ of $w_{S'}w_{S''}$.
Now $w_{S'}w_{S''} = 1.1w_S$, where $S = S' \circ S''$ is some subtrajectory
from $u$ to some vertex $x$ visited by $T$ of some length $k\leq 2\log n$.
Note that $S$ moves in the interval $[t-1,t]$ by (b), so there are $O(\log n)$ choices for $t$.
There are also $O(\log n)$ choices for each of $x$ and $k$,
so the total of such $w_{T'}w_{T''}$ is $O(\log^3 n)$.

The remaining cases for $(T',T'')$ are similar and also contribute $O(\log^3 n)$.
Indeed, suppose next that again (a) $T''$ meets $T'$ before it meets $T$,
but now (b') $T'$ moves in the time interval immediately before it first meets $T''$.
Now there are $O(n\log n)$ choices for $t$ (defined as before), but we sum over 
$w_{S'}w^*_{S''} \sim 1.1w_S/n$, so the bound is again $O(\log^3 n)$.
The final case is (a') $T''$ meets $T$ before it meets $T'$.
Here we consider $S = S' \circ S''$ where 
$S'$ is the subtrajectory of $T'$ from $u$ to its first meeting with $T''$
and $S''$ is the reverse of the subtrajectory of $T''$
from its first meeting with $T$ to its first meeting with $T'$.
This case is simpler than (a), as $S$ must move in the interval preceding 
the meeting of $T'$ and $T''$, and again the bound is $O(\log^3 n)$.

In conclusion, $D_{v,T}=O(\log^3 n)$, 
so $\Var(X)=\widetilde{O}(\mean{X})=o(\mean{X}^2)$, as required.
\end{proof}

\section{Hitting times and toppling} \label{sec:topple}

In this section we analyse two problems on annihilation: 
the worst-case expected extinction time (Proposition \ref{prop:worst}) 
and the particle-hole model (Theorem \ref{thm:p=0} on stationary reds). 
Both results use results on hitting times: 
for the former we use a result of Coppersmith, Tetali and Winkler \cite{CTW}
on meeting times for adversarially controlled random walks;
for the latter we prove a lower bound 
on the hitting time of a small set from a random position.
For the particle-hole model, we also need an interpretation of the model 
as an `abelian sandpile', meaning that the order in which blues are `toppled' does not matter,
and to devise a method for overcoming dependencies between their initial positions.

For our bounds on the worst-case expected extinction time,
we consider adversarial meeting times, defined as follows.
Let $G$ be a graph. Start with two particles, one at some vertex $x$
and another at some vertex $y$. A \emph{strategy} $S$ is a possibly random rule
that decides for each possible history of the process which particle 
should take a random walk step. The meeting time $M_S(x,y)$
is the expected number of steps until the tokens meet.
We use the following consequence of \cite[Theorem 2]{CTW}:
\begin{equation} \label{eq:ctw}
M_S(x,y) \le 2H_{\max}(G).
\end{equation}

\begin{proof}[Proof of Proposition \ref{prop:worst}]
Fix any graph $G$. For the lower bound, we consider $p=0$, \ie stationary reds.
Let $x$ and $y$ be vertices such that $H_{\max}(G)=H_x(y)$.
We consider the starting configuration with 
$k$ blue particles on $x$ and $k$ red particles on $y$. 
Then the extinction time is the sum of $k$ variables 
each with expectation $H_x(y)$, so has expectation $k H_{\max}(G)$.

For the upper bound, we consider an arbitrary partition
of the particles into $k$ oppositely coloured pairs.
For each pair, we count the total number of steps taken by that pair 
until at least one of them is destroyed. This is bounded by a meeting time $M_S(x,y)$,
where one particle starts at $x$ and one at $y$, with the strategy that at each time step
red moves with probability $p$ and blue moves with probability $1-p$.
By \eqref{eq:ctw} we have $M_S(x,y) \le 2H_{\max}(G)$.
After each annihilation, if this involves the destruction of two particles from different pairs 
then we create a new pair consisting of the two remaining particles from the previous pairs. 

Since the number of particles decreases by $2$ each time a new pair is created, 
at most $k-1$ new pairs are created, and so there are at most $2k-1$ pairs in total. 
Each step in the process is counted by some pair,
and the expected number of steps taken by each pair is at most $2H_{\max}(G)$,
so the  expected extinction time is less than $4kH_{\max}(G)$.
\end{proof}

Now we consider the case of stationary reds, i.e~$p=0$.
The abelian property for this case is discussed in \cite{CRS14}, 
using an alternative `site-wise randomness' construction of the process. 
Each vertex is first equipped with a fixed list of instructions,
where each instruction moves a particle to some given adjacent vertex,
and each instruction is sampled independently at random.
We `topple' a site by performing the first unused instruction in its list. 
A sequence $\alpha$ of sites is a legal toppling sequence
if each instruction can be implemented,
\ie it topples a site with at least one blue particle.
A toppling sequence reaches equilibrium if it removes all blue particles. 
For a vertex $v$, we write $m_\alpha(v)$ for the number of times $v$ appears in $\alpha$. 
The following is rephrased from \cite[Lemma 1]{CRS14}; for further details see \cite[Section 3]{RS12}.
\begin{lem}\label{lem:abelian}
Fix an initial configuration. If $\alpha$ and $\beta$ are both legal toppling sequences 
that reach equilibrium then $m_\alpha(v)=m_\beta(v)$ for every vertex $v$.
\end{lem}
In particular, the total number of movements does not depend on the order in which particles move, 
and so it suffices to prove the lower bound for some choice of toppling sequence.

We also require the following lower bound on hitting times. 

\begin{lem}\label{lem:hit-from-random}
Suppose $A$ and $B$ are disjoint non-empty vertex sets in a regular graph $G$. 
Let $H_A(B)$ be the first time a random walk, started from a uniformly random vertex of $A$, reaches a vertex of $B$. 
Then $\probb{H_A\geq \frac{|A|}{2|B|}}\geq 1/2$ and $\mean{H_A(B)}\geq \frac{|A|+|B|}{2|B|}$.
\end{lem}
\begin{proof}
Consider a stationary random walk $X_0,X_1,\ldots$ on $G$,
\ie a random walk starting from a uniformly random vertex,
which is therefore at a uniformly random vertex at each future step.
For any positive integer $k$, the expected number of visits to $B$
by $X_1,\dots,X_k$ is $k|B|/|V(G)|$. On the other hand,
this is at least the probability that there is some such visit,
which is at least  $\prob{X_0\in A}\prob{H_A(B)\leq k}$,
so we deduce $\prob{H_A(B)\leq k} \le k|B|/|A|$.

Taking  $k=\floorfrac{|A|}{2|B|}$ proves the first statement.

For the second statement, noting that 
$\prob{H_A(B)=0}=0$ as $A \cap B = \varnothing$, we have
\begin{align*}\mean{H_A(B)}&=\sum_{k\geq 0}\prob{H_A(B)>k}
\geq\sum_{k=0}^{\floor{|A|/|B|}}(1-k|B|/|A|)\\
&=(1+\floor{|A|/|B|})(2-\floor{|A|/|B|}|B|/|A|)/2.\end{align*}
Writing $x = |A|/|B|-\floor{|A|/|B|} \in [0,1)$ we calculate
\[ 2|B|\mean{H_A(B)} \ge (|A|+|B|-x|B|)(1+x|B|/|A|) 
= |A|+|B| + x(1-x)|B|^2/|A|. \]
This gives the required bound.
\end{proof}

We will conclude this section by proving Theorem \ref{thm:p=0} on stationary reds.
To illustrate the idea of the proof, one can consider a variant model
where the reds are positioned at some arbitrary set $S$ of $n/2$ vertices
and the blues each independently start at a random vertex not in $S$.
We can construct a toppling sequence by considering each blue particle 
in turn and revealing its starting location and random walk steps until it hits a red
particle not hit by a previous blue. The expected extinction time is thus the sum 
of hitting times $H_A(B)$ as in Lemma \ref{lem:hit-from-random},
with $|A|=n/2$ and $|B|=n/2,n/2-1,\dots,1$, which is $\Omega(n\log n)$.
For the actual model where we start one blue from each vertex not in $S$
we will need a further device to overcome dependencies.

\begin{proof}[Proof of Theorem \ref{thm:p=0}]
We show how to randomly construct a legal toppling sequence with at least the required  length, 
which is sufficient by Lemma \ref{lem:abelian}. For ease of writing, we assume that $\floor{n/2}$ is even.

First we reveal the locations of half the blues, and the empty vertex if $n$ is odd, giving a set $S_1$ with $|S_1|=\ceil{n/4}$. 
We build up a set $S_2$, initially empty, as follows. We consider each blue particle in $S_1$ in turn and reveal its random
walk steps until it  reaches $V(G)\setminus(S_1\cup S_2)$, then we add its current location to $S_2$.
This gives a legal sequence, since every vertex in $S_1$ either has one blue or is unoccupied, 
and every vertex in $S_2$ either has two blues (if the particle originally there was blue) 
or is unoccupied (if the particle originally there was red, and was annihilated by the blue that reached it).
After these steps we have a partition $V(G)=S_1\cup S_2\cup U$, 
where $S_2\cup U$ originally contained $\floor{n/4}$ blue and $\floor{n/2}$ red particles, 
$|S_2|=\floor{n/4}$ and $|U|=\floor{n/2}$. 

Now we reveal the number $k$ of blue particles that started in $S_2$,
thus also revealing that $\floor{n/4}-k$ blue particles started in $U$,
$\floor{n/4}-k$ red particles started in $S_2$ and $\floor{n/2}-\floor{n/4}+k$ red particles started in $U$.
We note that the above construction of $S_2$ was independent of the initial distribution of particles in $S_2\cup U$, 
so with high probability we have $k=(1/12+o(1))n$, and given the value of $k$, the vertices of $S_2$ that initially contained blue particles 
form a uniformly random $k$-subset of $S_2$. After following the toppling sequence above, each of these $k$ vertices 
contains two blue particles, and the remaining vertices in $S_2$ are empty. 

Next we reveal the colours of all particles in $U$ and for each blue particle in $U$ in turn, move that particle until it hits a red particle in $U$.
This gives a legal sequence because $U$ contains more red than blue particles and 
every particle in $S_1\cup S_2$ is either unoccupied or contains two blue particles.
We also note that this process of eliminating particles in $U$ is independent of the locations of the $k$ vertices with two blue particles. 
After this toppling sequence, we have a partition $S_1\cup S_2\cup U_1\cup U_2$ where $S_1\cup U_1$ is unoccupied, 
$|S_2|=\floor{n/4}$ and $k$ uniformly random vertices of $S_2$ have two blue particles each, 
and $|U_2|=2k$ with one red particle at each vertex in $U_2$.

Now, as in the illustration before the proof, we repeat the following procedure.
At each step, we reveal an occupied vertex $v$ in $S_2$ uniformly at random,
move one blue particle at $v$ until it hits a red particle not hit by a previous blue,
then do the same for the other blue particle at $v$. We remove $v$ from $S_2$ and
the two hit vertices from $U_2$, then continue to the next step.

Note that at each step we have $|S_2|>n/4-k=(1-o(1))n/6$. 
As $v$ is uniform in $S_2$, by Lemma \ref{lem:hit-from-random}
both of the blue particles at $v$ hit $U_2$
in expected time at least $|U_2|^{-1} (1-o(1))n/12$.
By linearity of expectation, we obtain  
$\mean{T}\geq(1-o(1))n\sum_{i=1}^{n/12}(12i)^{-1}>0.08n\log n$ for $n$ sufficiently large. 

Furthermore, by Lemma \ref{lem:hit-from-random}, the hitting time for the first particle of a pair 
is at least $(1-o(1))n|U_2|^{-1}/12$ with probability at least $1/2$, 
independently of earlier hitting times, so we can write $T \ge (1-o(1))n\Sigma$,
where $\Sigma = \sum_{i=1}^{n/12}X_i(24i)^{-1}$ with the $X_i$ being 
independent Bernoulli random variables with parameter $1/2$. 
Now $\Sigma$ has expectation $\frac{1}{24}\log n-O(1)$ and variance $O(1)$,
so by Chebyshev's inequality with high probability  $T>0.04n\log n$.
\end{proof}

\section{Few predators} \label{sec:predator}

We conclude this part of the paper by proving Theorem \ref{thm:predator}
on the predator-prey model with few predators.
 
\begin{proof}[Proof of Theorem \ref{thm:predator}]
We start with the lower bound. Similarly to the proof of Theorem \ref{thm:persistent},
we will consider a parallel system of non-interacting `fake' particles,
where the fake particles start from, and hence remain at, independent uniformly random positions.
We will have $2k$ fake predator particles and $n-k$ fake prey particles,
so that with high probability we can couple the starting distributions of true and fake particles
via a partial pairing between particles of the same types in the same location,
so that every true predator particle is paired with some fake predator particle
and $n/2$ true prey particles are each paired with fake prey particles. Here it is necessary to have significantly more fake particles of each type than the number of true particles we intend to pair, in order that the pairing succeeds with high probability despite the different starting distributions of true and fake particles. For the lower bound, we need all true predator particles to be paired, so that unpaired particles cannot destroy paired particles, but we do not need all prey particles to be paired.
The predator pairs remain paired throughout the process, whereas a fake prey particle
becomes unpaired if its paired true particle is eliminated.
In each step, we select a particle to take a random walk step 
with probability proportional to its speed,
where we simultaneously consider true and fake particles, 
but count paired particles as a single particle.
If a pair is selected then they take the same random walk step in both systems.
If an unpaired particle from one system is selected then the other system ignores this step.

We note that a paired true prey particle is destroyed only if its paired fake prey particle
meets a paired fake predator particle (since all true predator particles are paired).
We define the `fake time' at any step as the total number of movements by fake particles
and let $X_t$ be the number of meetings between opposite-type fake particles by fake time $t$.
Writing $s_A,s_B$ for the speeds of predators and prey respectively, where $s_A+s_B=1$, 
at any given step when a fake particle moves, 
as fake particles are independent and uniformly distributed, 
the probability of $X_t$ increasing is at most
\[\frac{ks_A\frac{n-k}{n}+(n-k)s_B\frac{k}{n}}{ks_A+(n-k)s_B}
= \frac{k(n-k)/n}{ks_A+(n-k)s_B} \leq\frac{k}{ns_B}.\]

Let $\tau$ be the stopping time with $X_\tau = n/4$.
Note that the true process has not yet reached equilibrium,
as at least $n/4$ paired true prey particles have not been eliminated.
As $\mean{X_t} \le \frac{kt}{ns_B}$, we have
$\prob{\tau \le t} \le \prob{X_t \ge n/4} \le \frac{4kt}{n^2 s_B}$ by Markov's inequality, so
\[\mean{\tau}=\sum_{t\geq 0}\prob{\tau>t}
\geq \sum_{t=0}^{\floor{n^2s_B/4k}}\left(1-\frac{4tk}{n^2s_B}\right)
\geq \frac{n^2s_B}{8k}.\]

Write $M_t$ be the number of movements by true particles by fake time $t$.
Then $M_\tau$ is a lower bound on the equilibrium time.
To relate $M_\tau$ to $\tau$, we note that for any $t<\tau$, with probability at least $1/4$
the fake movement at fake time $t$ is paired with a corresponding true movement.
Indeed, there are $k$ paired predator particles and at least $n/4$ paired fake prey particles,
so the required probability is at least $\frac{ks_A  + s_B n/4}{2ks_A + s_B (n-k)} \ge 1/4$.
It follows that $M_t-t/4$ is a submartingale for $t\leq \tau$, 
so $\mean{M_\tau}\geq \mean{\tau}/4 =\Omega(n^2/k)$.
 
Now we prove the upper bound, arguing similarly to the proof of Theorem \ref{thm:persistent}.  
We divide into rounds, where within each round
with high probability at least half of the remaining prey particles are destroyed.
At the start of each round we couple to a system of fake non-interacting particles,
now starting each round with a bijective pairing as in the proof of Theorem \ref{thm:persistent}.
Again, predator particles remain paired, but prey particles can become unpaired.
The true and fake populations differ by at most a factor of two 
before half of the  remaining prey particles are destroyed.

In each round $i$, writing $m_i$ for the number of remaining particles, we consider a mixing phase 
of $O(m_i \log n)$ steps and then $O(m_in/k)$ additional steps for destroying prey particles.
Continuing to follow the proof of Theorem \ref{thm:persistent}, 
with high probability each fake particle takes $\Theta(\log n)$ steps in the mixing phase,
and conditional on the initial locations of the fake particles, 
and the number of random walk steps each has taken, all being $\Theta(\log n)$,
their locations are independent random variables
and approximately uniform in the sense of \eqref{prob-bounds}. 

We claim that with high probability
at each time $t$ in round $i$ with $\Omega(m_i \log n) = t = O(n^2)$,
the predator particles occupy at least $k/4$ positions.
Indeed, if this fails at some time $t$ then when revealing the positions of predator particles 
one by one there are at least $3k/4$ times when we reveal a particle in one of at most $k/4$ occupied positions.
Any particle has probability at most $2/n$ of being in some given position,
so this event has probability at most $2^k(k/2n)^{3k/4} = o(n^{-2})$ for $k \ge 4$.
The claim follows by a union bound.

Now we apply Lemma \ref{lem:hitting} to each prey particle $p$ in turn,
with the moving target sets $A_i$ corresponding to $k/4$ locations of predator particles.
For a suitable choice of constants, with high probability each such $p$ takes at least $12rn/k$ steps,
so by Lemma \ref{lem:hitting} hits a predator particle with probability at least $2/3$, say,
so with high probability at least half the prey particles are destroyed in this round. 
We may rerun any unsuccessful rounds without changing the expected time by more than a $1+o(1)$ factor.
Since $m_i\leq 2^{1-i}(n-k)+k$ and $O(\log n)$ rounds are needed, 
the total expected time is $O(n\log n+k\log^2 n)=O(n\log n)=O(n^2/k)$ for mixing 
and $O(n^2/k+n\log n)=O(n^2/k)$ for hitting, as required.
\end{proof}

\newpage

\part{Annihilation}

In this second part of the paper we focus on the balanced two-type annihilation model,
for which the lack of persistence poses several additional challenges 
not seen when analysing the models in the first part of the paper.
Our main result here is Theorem \ref{thm:main}. We divide the proof into two sections,
presenting the upper bound in Section \ref{sec:upper} and the lower bound in Section \ref{sec:lower}. 

\section{Annihilation upper bound}\label{sec:upper}

In this section we prove the upper bound for Theorem \ref{thm:main}. 
While our main interest is in the two-type model, our proof also applies to the one-type model,
and sometimes the one-type model is useful for giving a simpler illustration of the proof ideas
(the main simplification is that here all particles move at the same speed).
We deduce it from the following stronger result, where we do not need any specific lower bound on the spectral gap, 
but get a constant that depends on the spectral gap. 

\begin{thm}\label{thm:expander}
For any $\mu>0$ there is $C=C(\mu)>0$ such that the following holds.
Let $G$ be a regular graph on $n$ vertices with $1-\mu_2(G)>\mu$. 
Consider one-type annihilation or balanced two-type annihilation with arbitrary speeds
from an arbitrary valid initial configuration. Let $T$ be the extinction time.
Then $T \le Cn\log n$ with high probability and in expectation.
\end{thm}

It suffices to focus on the high probability statement in Theorem \ref{thm:expander}.
Indeed, by increasing $C$ we will see that we can get any polynomial failure probability.
To deduce the expectation statement  in Theorem \ref{thm:expander}, we combine this
with the bound from Proposition \ref{prop:worst} that the expected extinction time from any
configuration is  $O(nH_{\max}(G)) = O(n^2)$, as the maximum hitting time of a regular expander 
is $\Theta(n)$ (see e.g.\ \cite[Corollary 3.3]{Lov96}).

\subsection{A Hall matching argument}

Similarly to the proof of the upper bound in Theorem \ref{thm:predator},
the overall plan for the proof of the upper bound in Theorem \ref{thm:expander}
is to divide into rounds, where in each round $1/4$ of the remaining particles are destroyed.
Again we will use the Moving Target Lemma,
but we cannot use the simple sequential argument as for predator-prey,
as for annihilation the targets disappear as they are hit by other particles.

Instead, our plan will be to use a Hall matching argument to find a large collection
of (vertex-)disjoint pairs of oppositely coloured meeting fake particles, arguing that 
this implies that many annihilations happen in the true process.
Suppose there are $k$ particles of each type remaining. 
As before, we couple to a system of non-interacting fake particles,
where initially there is a bijective pairing of true and fake particles.
Then we have the following simple fact.
\begin{lem}\label{lem:non-interacting}
Suppose that at least $\ell$ disjoint pairs of oppositely coloured fake particles 
have met after $t$ fake steps. Then at least $\ell/2$ particles 
of each type have been destroyed after at most $t$ true steps.
\end{lem}
\begin{proof}
For any $i\leq t$, the first $i$ fake steps correspond to $j$ true steps for some $j\leq i$. 
If two fake particles of opposite colours reach the same vertex after $i$ fake steps, 
then either they annihilate one another within $j\leq i \le t$ true steps, 
or at least one of them has been destroyed before that point. 
Thus we can choose at least one particle from each pair which is destroyed within $t$ true steps.
\end{proof}

To find the disjoint pairs assumed in Lemma \ref{lem:non-interacting}, 
we will consider an auxiliary bipartite graph
where edges represent meeting pairs of oppositely coloured fake particles,
to which we will apply the following version of Hall's theorem
(see e.g. \cite[III.3, Corollary 9]{Bollobas}).

\begin{thm} \label{dhall}
Let $H$ be a bipartite graph with parts $A$ and $B$.
Suppose that any $S \subseteq A$ has at least $|S|-d$ neighbours in $B$.
Then $H$ contains at least $|A|-d$ disjoint edges.
\end{thm}

We now come to the main lemma of this subsection,
in which we apply Hall Matching and the Moving Target Lemma
to find disjoint pairs as in Lemma \ref{lem:non-interacting}, 
provided the red particles are sufficiently spread out 
and the blue particles move sufficiently many times.

\begin{lem}\label{lem:halls}
Let $G$ be a regular graph on $n$ vertices, $0 < c \le 1/4$
and $r,s,k$ be positive integers with $k \le n/2$ and $\mu_2(G)^{2r} < 1/17$.
Consider a system of at least $k$ red particles and at least $k$ blue particles
all following non-interacting lazy random walks on $G$, with one particle moving at each time step.
Suppose the initial positions $X_i^{(0)}$ of the blue particles are independent and satisfy
	\[\sum_{v\in V(G)}(\prob{X_i^{(0)}=v}-1/n)^2\leq \frac{ck}{4n^2}. \]
Then with probability at least $1-0.8^{sk/4}$ there are
	\begin{enumerate}[(i),nosep]
		\item at most $3k/4$ blue particles that move at least $6nrs/(ck)$ times, or
		\item at least $k/2$ red particles at some time occupying a set of fewer than $ck$ vertices, or
		\item at least $k/2$ meetings of disjoint pairs of oppositely coloured particles.
	\end{enumerate}
\end{lem}

\begin{proof}
We can assume that there are exactly $k$ particles of each colour,
by restricting attention to an arbitrary set of $k$ red particles
and the $k$ blue particles with the most movements.
We reveal the times at which each particle moves and the position of the red particles at each time. 
By independence of the initial positions of the walks,
the positions of the blue particles at each time are independent
conditional on the times at which each walk moves. 
Thus the events of different blue particles hitting any subset of the red particles are independent
conditional on the positions of the red particles at each time. 

We will bound the probability that events (i), (ii) and (iii) all fail.	
To do so, we estimate the probability of the event (iii') that one cannot find
a subset $A$ of $k/2$ red particles and a subset $B$ of $k/2$ blue particles
with no meeting between a particle in $A$ and a particle in $B$.
By Theorem \ref{dhall} event (iii') implies event (iii).
We will show that when events (i) and (ii) fail
then it is very unlikely that event (iii') fails.

Fix some $A$ and $B$ as above.
As event (ii) fails, at each time $i$  we can fix a `target' set $A_i$ of $ck$ vertices 
which are occupied by particles in $A$ at time $t'$. 
We will apply Lemma \ref{lem:hitting} with these target sets (replacing $k$ by $ck$),
which are likely to hit by particles in $B$ moving at least $\ell rs = 6nrs/(ck)$ times.
As event (i) fails, the subset $B'$ of such particles in $B$ has size $|B'| \ge k/4$.
Applying Lemma \ref{lem:hitting}, each blue particle in $B'$ hits some $A_{t'}$ with probability at least $1-\ee^{-3s}>1-20^{-s}$. 
Thus the probability that the particles in $A$ and $B$ do not meet is at most $20^{-sk/2}$. 
Since there are at most $4^k$ choices for the sets $A$ and $B$, 
the failure probability is at most $4^k/20^{sk/2}\leq 0.8^{sk/2}$. 
\end{proof}
	
In order to apply Lemma \ref{lem:halls}, we will repeatedly need to show that most of the blues move sufficiently many times in some period.

\begin{lem}\label{lem:blues-move}
Suppose there are $k\leq n/2$ particles of each type.
Fix $r'\geq 1$. Consider $96nr'$ steps of the fake process.
Then with probability $1-o(n^{-4})$,
at least $3k/4$ blue particles each move at least $24nr'/k$ times.
\end{lem}
\begin{proof}
We will bound the probability that there is some set of $k/4$ blue particles
where each moves at most $24nr'/k$ times.
There are at most $\binom{n/2}{n/8}\leq (4\ee)^{n/4}$ sets of $k/4$ blue particles. 
The total number of movements made by such a set is binomially distributed with mean at least $12nr'$, since $p\geq 1/2$. 
By a standard Chernoff bound, this is at least $6nr'$ with probability at least $1-\exp(-3n/2)$. 
If it is, then at least one of those particles moves at least $24nr'/k$ times. 
The result follows from a union bound, since $(4\ee)^{n/4}\exp(-3n/2)=o(n^{-4})$.
\end{proof}

We conclude this section by recording a bound on how many fake steps are needed in the mixing phase of each round.
Depending on the parameters, we may have enough time to mix all particles or only time to mix the faster blue particles.

\begin{lem}\label{lem:mixing2}
Fix $\alpha>0$, let $\beta=\beta(\alpha)>1$ with $\alpha(\beta-\log\beta-1)= 4$, and define 
\begin{equation}t_{k}^{\mathrm{blue}}:=\ceil{2\alpha\beta k\log n}
\quad\text{and}\quad t_{k}^{\mathrm{all}}:=\ceil{ p^{-1}\alpha\beta k\log n}.\label{t1}\end{equation}
Consider the fake process  with $k\leq n/2$ particles of each colour. 
Then with probability at least $1-n^{-3}$,
\begin{enumerate}[(i),nosep]
\item by time $t^{\mathrm{blue}}_k$
each blue particle has moved at least $\alpha\log n$ times,
\item by time $t^{\mathrm{all}}_k$
every particle has moved at least $\alpha\log n$ times.
\end{enumerate}
\end{lem}

\begin{proof}
We consider a union bound of probabilities that some given particle moves fewer than $\alpha\log n$ times.  
This is at most the probability that a binomial random variable with mean $\mu\geq \alpha\beta\log n$ (in the first case, using $1-p\geq 1/2$) 
attains a value less than $\mu/\beta$. By a standard Chernoff bound, using the definition of $\beta$, 
this has probability at most $(\ee^{1/\beta-1}\beta^{1/\beta})^\mu\leq n^{-4}$. 
\end{proof}
When applying Lemma \ref{lem:mixing2} we shall always take 
the specific value $\alpha=-3/\log \mu_2=\Theta(1)$ required 
to achieve approximate uniform distribution as in \eqref{prob-bounds}.

\subsection{Bounded speed ratio}\label{sec:p-const}

Here we prove Theorem \ref{thm:expander} under the simplifying assumption of bounded speed ratio,
\ie the red particles move at a speed $p \in [c,1/2]$ for some fixed $c>0$.
This allows us to give a streamlined proof showcasing the main technique,
to be followed by a full proof for all $p$ in Section \ref{sec:any-p}.

As discussed above, our basic strategy is to divide into rounds, 
in each of which we eliminate $1/4$ of the remaining particles.
Given a round that starts with $k$ particles of each type,
we will use our usual coupling of true and fake particles
to bound the time taken by the round by the time for 
finding $k/2$ disjoint opposite-coloured meeting pairs of fake particles;
this is sufficient by Lemma \ref{lem:non-interacting}
and achievable by Lemma \ref{lem:halls}.

We start with some basic reductions. We can assume that $n$ is sufficiently large.
We can assume that at the start of any round $k$ is at least some large constant, say $100$.
Indeed, consider the process when only $100$ particles of each colour remaining. 
Then we may mix all blues with high probability in time $O(\log n)$,
then wait until some blue has moved $6rn\log n$ times,
which takes time $O(n\log n)$ with high probability. 
By Lemma \ref{lem:hitting} with $k=1$, this blue hits any fixed red particle with probability at least $1-n^{-3}$. 
Repeating $100$ times eliminates all remaining particles in time $O(n\log n)$ with suitable probability. 

To bound the time taken by any given round, we think of it as broken up into phases:
first a `mixing phase' of $t^{\mathrm{all}}_k$ fake steps, as defined in Lemma \ref{lem:mixing2}, 
followed by a `hitting phase' of $96nr$ fake steps, with $r$ as in Lemma \ref{lem:halls}.
The round ends as soon as $k/4$ particles 
of each colour have been destroyed (which may happen during the mixing phase).
By Lemma \ref{lem:mixing2}, if the mixing phase runs to completion then 
with probability at least $1-n^{-3}$ mixing `succeeds',
meaning that every particle moves at least $\alpha\log n$ times.

Assuming that mixing succeeds, we reveal the positions of the blues before the mixing phase
and the number of times each has moved within it. Then the conditional distribution of the blues
is independent with each being approximately uniform as in \eqref{prob-bounds},
which is easily good enough to apply Lemma \ref{lem:halls}. This tells us that
during the hitting phase, with probability at least $1-0.8^{k/4}$ there are
	\begin{enumerate}[(i),nosep]
		\item at most $3k/4$ fake blue particles that move at least $24nr/k$ times, or
		\item at least $k/2$ fake red particles at some time occupying a set of fewer than $k/8$ vertices, or
		\item at least $k/2$ meetings of disjoint pairs of oppositely coloured fake particles.
	\end{enumerate}
Here (iii) is our desired outcome for hitting to `succeed', 
and we can assume that (i) does not hold by Lemma \ref{lem:blues-move}. 
Thus we need to bound the probability of (ii), as follows.

\begin{lem} \label{no-ii}
At any step in the hitting phase, the probability that there is some set of $k/2$ fake red particles
occupying fewer than $k/8$ vertices is at most $(1.5k/n)^{3k/8}$.
\end{lem}

\begin{proof}
By approximate uniformity as in \eqref{prob-bounds},
a union bound over $\binom{k}{k/2}$ choices for the set of particles
and $\binom{n}{k/8}$ choices for the set of vertices gives
\begin{align*} & \binom{k}{k/2}\binom{n}{k/8}
        \left( \frac{k}{8} \left( \frac{1}{n} + \frac{1}{n^3}  \right) \right)^{k/2}\\
	&\leq (1+O(k/n^2)) 2^k \left(\frac{8en}{k}\right)^{k/8}\left(\frac{k}{8n}\right)^{k/2}\\
	& \le 2 \left(\frac{ek^3}{2n^3}\right)^{k/8} \leq (1.5k/n)^{3k/8}, \text{ as } 100 \le k\leq n/2. \qedhere
\end{align*}
\end{proof}

For $100 \le k\leq n/2$, the bound in Lemma \ref{no-ii} is maximised at $k=100$,
and is easily $O(n^{-50})$, say. Thus with suitably high probability 
mixing and hitting succeed throughout the process. There are $O(\log n)$ rounds, 
so the total time spent in hitting phases is  at most $O(n\log n)$.
The total time spent in mixing phases is bounded by the sum of 
$t_{k}^{\mathrm{all}}:=\ceil{ p^{-1}\alpha\beta k\log n}$
for $k = (3/4)^i n$, $i \ge 0$. This is also $O(n\log n)$,
by our assumption that $p$ is bounded away from zero.
This completes the proof of  Theorem \ref{thm:expander} in this case.

\subsection{General speeds}\label{sec:any-p}

Now we will prove Theorem \ref{thm:expander} in full generality.
Our approach is similar to that for the case of bounded speed ratio,
except that we may not have time to mix the slower red particles,
so we need some other way to argue that they are well-distributed.
We consider up to three separate regimes according to the number of remaining particles.
In the first `dense' regime, the red particles will be well-distributed because most of them 
have not had enough time to move from their starting locations.
In the third `sparse' regime, there are so few particles that we have time to mix the reds.
In the second `intermediate' regime, neither of these approaches works, but as it is relatively narrow  
we can afford a worse distribution of the red particles.

We define these regimes in terms of the number $k$ of remaining particles in each colour, as follows. 
The first regime covers all $k$ greater than $k^* := \max\{ c_1 pn\log n,20\log n\}$, 
for some constant $c_1$ depending on the spectral gap of $G$ to be defined below,
the third regime covers all $k\leq pn$, and the second regime covers all intermediate values.
We note that over the second regime $k$ varies by a factor at most $c_1\log n$ (since we shall have $c_1\geq20$).

We start by indicating how the third regime is handled by the same proof as in the previous subsection.
Recall that in the round starting with $k$ particles of each colour, we reduce from $k$  to $3k/4$,
analysing the time taken by considering a mixing phase of fake time $t^{\mathrm{all}}_k$ 
followed by a hitting phase of fake time $96nr$. We showed that with suitably high probability 
mixing and hitting succeed throughout the process. There are $O(\log n)$ rounds, 
so the total time spent in hitting phases is  at most $O(n\log n)$.
The total time spent in mixing phases is bounded by the sum of 
$t_{k}^{\mathrm{all}}=\ceil{ p^{-1}\alpha\beta k\log n}$
for $k = (3/4)^i pn$, $i \ge 0$, which is also $O(n\log n)$, as required.
Thus with high probability in the third regime extinction occurs in time $O(n\log n)$.

For the first regime, we use the same analysis, 
except that we can only allow a mixing phase of fake time $t^{\mathrm{blue}}_k$.
This ensures that the total time spent in mixing phases is bounded by the sum of 
$t_{k}^{\mathrm{blue}}=\ceil{2\alpha\beta k\log n}$
for $k = (3/4)^i n$, $i \ge 0$, which is $O(n\log n)$.
Now success of the mixing phase means that every fake blue particle moves $\alpha n$ times. 
The previous proof will still show that mixing and hitting succeed with high probability
if we can find a replacement for Lemma \ref{no-ii}. This is achieved by the following lemma,
which implies that with high probability throughout the first regime 
there is no set of $k/2$ red particles occupying a set of fewer than $k/4$ vertices. 

\begin{lem}\label{lem:not-many-move}
The probability that more than $k^*/4$ red particles move in time $\frac{c_1}{20}n\log n$ is at most $n^{-4}$.
\end{lem}
\begin{proof}
A standard Chernoff bound gives that for a binomial random variable with mean at most $a/5$ 
has value exceeding $a$ with probability at most $(\ee^{4/5}/5)^a<\ee^{-4a/5}$. 
Now the number of movements of red particles in $\frac{c_1}{20}n\log n$ time steps is binomial, 
with expectation $\frac{c_1}{20}pn\log n\leq k^*/20$. 
Thus this exceeds $k^*/4$ with probability at most $\ee^{-k^*/5}$, 
and since also $k^*\geq 20\log n$, this gives the desired bound.
\end{proof}

Choosing $c_1\geq 20(8\alpha\beta+96 r/\log(4/3))$ we see that in time $\frac{c_1}{20}n\log n$ 
the first regime succeeds with high probability, meaning that we reduce $k$ to $k^*$.

\subsubsection{The intermediate regime}\label{sec:second-regime}

The main challenge of the intermediate regime is that we cannot rule out
the admittedly implausible scenario that the slow red particles huddle together
on a few sites that are hard for the blue particles to hit.
We start with a lemma showing that at least  this scenario cannot be too extreme.
(A similar bound appears for the complete graph in \cite{2-type}, 
but their argument is specific to that case.) 

\begin{lem}\label{lem:poisson2}
Position one red particle at every vertex of a regular graph of order $n$.
Consider a process where in each step an arbitrary (possibly empty) set of red particles is removed
and then a uniformly random particle takes a random walk step. Then with probability $1-O(n^{-3})$, 
no vertex has more than $8\log n/\log\log n$ red particles at any point in the first $n^2$ steps.
\end{lem}

The intended application of this lemma is that at time $0$
all of the `red' particles that should be blue are removed,
and thereafter red particles will be removed by annihilation 
in some random way for which we consider an adversarial model.
It is an immediate consequence of the following lemma,
as by Chernoff bounds with high probability every particle
moves at least $n^2$ times in the first $n^3$ steps if we ignore removals,
so no matter how the removals are chosen these $n^3$ steps 
cover at least $n^2$ steps (or all the steps, if fewer).

\begin{lem}\label{lem:poisson}
Position one red particle at every vertex of a regular graph of order $n$.
Consider a process where in each step a uniformly random particle takes a random walk step.
Then with probability $1-O(n^{-3})$, no vertex has more than $8\log n/\log\log n$ red particles at any point in the first $n^3$ steps.
\end{lem}
\begin{proof}
We will find it more convenient to analyse a similar process with $n$ green particles independently positioned at uniformly random vertices. We say that a red particle is `covered' if at least one green particle was placed at its starting position; note that each red particle is uncovered with probability $(1-1/n)^n<1/\ee$. Pair each covered red particle with a green particle starting in the same position. We couple the green process to the red process as follows. If a covered red particle is selected to move at a given time step, select the corresponding green particle and move it to the same vertex. Otherwise, select and move an unpaired green vertex uniformly at random.
	
Suppose that $k$ red particles occupy a given vertex at a given time step. Since the expected number of these that are uncovered is at most $k/\ee$, Markov's inequality gives a constant probability that at most $k/2$ of them are uncovered, and hence that at least $k/2$ green particles occupy that vertex at that time. Thus it suffices to prove that the probability of some vertex having more than $8\log n/\log\log n$ green particles at any time step is $O(n^{-3})$.
	
Since the green particles are in the stationary distribution, for any vertex $v$ and time $t$
the number of green particles at $v$ at time $t$ has distribution $\operatorname{Bin}(n,1/n)$.
For any $r \ge 3$ we have
\[ \prob{\operatorname{Bin}(n,1/n)=r} < r!^{-1}(1-1/n)^n < \prob{\operatorname{Po}(1)=r}, \]
so applying  \eqref{poisson-bounds} with $r=8\log n/\log\log n$ and $\lambda=1$ gives the probability bound
\[\exp\left(\frac{-8\log n}{\log\log n}(\log\log n-\log\log\log n+\log 8)-1\right)=n^{-8+o(1)}.\]
The result follows from a union bound over all vertices $v$ and $t\leq n^3$.
\end{proof}

Now we proceed to the analysis of the process.
As in the first regime, in each round 
we have a mixing phase of fake time $t^{\mathrm{blue}}_k$.
Again mixing succeeds with high probability, so we can assume that 
every fake blue moves at least $\alpha n$ times in each round.
For the hitting phases, we do not attempt to show that 
each is small but instead focus on bounding their total time. 

We will apply Lemma \ref{lem:halls} to bound the hitting phases,
replacing our previous $c=1/4$ in (ii) by $c_k = \max\{\frac{\log\log n}{16\log n},1/k\}$.
This is justified by Lemma \ref{lem:poisson2}: we can assume that at no point 
do we ever have more than $\frac{8\log n}{\log\log n}$ red particles at a given vertex.
We say $k$ is `small' if $c_k=1/k$ and `large' otherwise.
A convenient unit for measuring time will be $B := 1920rn\log n/\log\log n$;
we will call a period of $B$ fake steps a `block'.
For (i), in any $s'$ blocks, by Lemma \ref{lem:blues-move} at least $3k/4$
fake blue particles each move at least $s'B/4k = \frac{480nrs'\log n}{k\log\log n}$ times.
If $k$ is large, by Lemma \ref{lem:halls} with $s=5s'$, the probability of the hitting phase  
lasting for more than $s'$ blocks is at most $0.8^{5s'k/4}< z^{s'}$, where
$z=\exp(\frac{-4\log n}{\log\log n})$. 
If $k$ is small, we instead apply Lemma \ref{lem:halls} with 
$s=\floor{\frac{80s'\log n}{\log\log n}}$, and easily have $0.8^{sk/4}<z^{s'}$
since $k\geq 20$.

In other words, the number of blocks in any round is stochastically dominated
by a geometric random variable with parameter $1-z$, \ie to bound the number of blocks in
a round we generate independent Bernoulli($z$) variables until the first $0$. 
Our bound for the total of the hitting phases is $B$ times the total number of Bernoulli variables 
needed across all rounds. Since at most $\log_{4/3}(c_1\log n)$ rounds take place, the probability 
of needing more than $m=\log_{4/3}(c_1\log n)+\log\log n$ Bernoulli variables is at most 
$\prob{\operatorname{Bin}(m,z)\geq\log\log n}$, which can be bounded by $2^{m}z^{\log\log n}=\widetilde{O}(n^{-4})$. 
Thus, with suitably high probability the total time for this regime is $O(B\log\log n)=O(n\log n)$, as required.

\section{Annihilation lower bound}\label{sec:lower}

In this final section of the paper, we conclude the proof of Theorem \ref{thm:main} by establishing the lower bound.

\subsection{One-type annihilation} 

To illustrate the main ideas in a simpler setting, we start by considering one-type annihilation, as follows.

\begin{thm}\label{thm:lower-bound-one}
Let $G$ be a regular graph on $n$ vertices with spectral gap more than $1/3$. 
Consider one-type annihilation on $G$ and let $T$ be the time until only one particle remains.
Then $T \ge cn\log n$ with high probability and in expectation, where $c$ is an absolute constant.
\end{thm}

The main action in the proof will take place while the number $k$ of remaining particles 
satisfies $n^{1-x}\geq k\geq n^x$ for some constant $x\in(0,1/2)$. 
The key step is to show that in this regime there are constants $c_1,c_2,c_3>0$
so that starting from any configuration $A$ of $k$ particles, 
with probability at least $c_1$ we need at least $c_2 n$ steps
to reduce the number of particles to $c_3 k$.
This will suffice to prove the result, as then with high probability
$\Omega(\log n)$ such intervals take $\Omega(n)$ steps to traverse. 
The proof of this key step has three main ideas:
\begin{enumerate}[noitemsep]
\item We consider a mixing phase (with the usual coupling of true and fake particles) 
and use the explicit spectral gap to argue that only a small constant proportion 
of the particles return to the starting set $A$ during the mixing phase.
\item Considering the survival during the mixing phase 
of particles not returning to $A$, we note that 
(a) not many can be annihilated by the few particles that do return to $A$, 
and show that
(b) annihilations between two particles both not returning to $A$ can be bounded
similarly to those returning to $A$ by a trajectory reversal argument.
\item Once we have enough surviving mixed particles we are essentially done, 
as then at each step collisions occur with probability $O(k/n)$, so with constant probability 
$\Omega(n)$ steps are required for $\Omega(k)$ collisions.
\end{enumerate}

Before starting to fill in the details of this sketch, 
we record some further facts on mixing and hitting for random walks.
We consider the lazy random walk $X_0,X_1,\dots$ on some multigraph on $n'$ vertices,
with stationary distribution $\pi$ and $1-\mu_2$ bounded away from zero.
\begin{enumerate}[(a),noitemsep]
\item The return time $H_v^+$ is the first positive time at which the walk started
at $X_0=v$ reaches $v$; it has mean $\mean{H_v^+}  = 1/\pi_v$ 
(see \cite[Lemma 5 of Chapter 2]{AFbook}).
\item  A strong stationary stopping time $\tau$ is a stopping time such that
$X_\tau$ is stationary and independent of $\tau$. By \eqref{prob-bounds} and
\cite[Lemma 24.7]{LP} there is such $\tau$ with $\mean{\tau}=O(\log n')$. 
By ignoring the first $\log^2 n'$ steps,
there is such $\tau$ with $\tau \ge \log^2 n'$ and $\mean{\tau}=O(\log^2 n')$.
\item The expected hitting time of $v$ starting from $\pi$ is
\begin{equation}\label{H-pi-Z}\mathbb{E}_\pi(H(v))=\pi_v^{-1}Z_{vv},
\text{ where } Z_{vv} := \sum_{t\geq 0}(P^t_{vv}-\pi_v) \le (1-\mu_2)^{-1}.
\end{equation}
Here the equality is \cite[Lemma 11 of Chapter 2]{AFbook},
and the inequality, noted in \cite[Lemma 3]{CEOR}, is immediate from 
$|P^t_{vv}-\pi_v| \le \mu_2^t$, which is the general form of \eqref{converge}.
\end{enumerate}

We also require some notation and a simple result for `collapsed chains'
(see \cite[Section 7.3 of Chapter 2 and Corollary 27 of Chapter 3]{AFbook}).
Given $A\subset V(G)$, we write $G/A$ for the multigraph obtained by contracting $A$
to a single vertex $v_A$. This is defined on the vertex set $V(G/A) = (V(G) \setminus A) \cup \{v_A\}$.
For each $x \in V(G)$ write $x_A = v_A$ if $x \in A$ or $x_A = x$ otherwise.
Then for each edge $xy$ of $G$ we have an edge $x_A y_A$ of $G/A$, 
included with multiplicity and allowing loops at $v_A$. The variational characterisation
of eigenvalues implies $\mu_2(G_A)\leq \mu_2(G)$,
\ie the spectral gap of $G_A$ is at least as good as that of $G$. 

Throughout the remainder of this subsection we fix $G$ as in Theorem \ref{thm:lower-bound-one}.
By decreasing $c$ we can assume that $n$ is sufficiently large.
We are now ready to implement the first part of the above sketch,
with the following definition and accompanying lemma.

\begin{defn} \label{pAT}
Given $A\subset V(G)$ and a time $T$,
we let $p_{A,T}$ be the probability that a lazy random walk on $G$
starting from the uniform distribution on $A$ visits $A$ before time $T$
after at least one non-lazy step.
\end{defn}

\begin{lem}\label{lem:no-return}
Fix $x>0$ and $A\subset V(G)$ with $|A|<n^{1-x}$. 
Then $p_{A,\log^2 n} \leq 2\mu_2-1+o(1)$.
\end{lem}

\begin{proof}
Let $H_A^+$ be the first return time to $A$ 
of a lazy random walk starting from the uniform distribution on $A$.
We consider $H_A^*$ defined in the same way 
except that we require at least one non-lazy step.
We claim that
\begin{equation}\label{HA*}
\mean{H_A^+} = n/|A| \quad \text{ and } \quad \mean{H_A^*}=1+2n/|A|.
\end{equation}
To see this, we consider the natural coupling of such a random walk to a random walk in $G/A$
started from $v_A$, observing that $H_A^+$ has the same distribution as $H_{v_A}^+$.
The stationary probability of $v_A$ is $|A|/n$, so $\mean{H_A^+} = \mean{H_{v_A}^+} = n/|A|$.
Writing $E$ for the event that the first step is lazy, we have
$\mean{H_A^+} = \mean{H_A^+ 1_E} + \mean{H_A^+ 1_{E^c}} = \frac{1}{2} + \mean{H_A^+ 1_{E^c}}$.
Also, $H_A^*$ is obtained by waiting for the first non-lazy step, which takes expected time $2$,
then adding $\mean{H_A^+ 1_{E^c} \vert 1_{E^c} } = 2\mean{H_A^+ 1_{E^c}}$.
Therefore $\mean{H_A^*} = 2 + 2\mean{H_A^+ 1_{E^c}} = 2 + 2(\mean{H_A^+}-1/2) = 1 + 2n/|A|$.

Now we let $\tau$ be a strong stationary stopping time
with $\tau \ge L := \log^2 n$ and $\mean{\tau}=O(L)$.
We consider the estimate
\[  \mean{H_A^*} = \mean{H_A^* 1_{H_A^* < L}} + \mean{H_A^* 1_{H_A^* \ge L}}
\le L + (1-p_{A,L})( \mean{\tau}+(1-\mu_2)^{-1}n/|A|), \]
where for the second term, given that the return time is larger than $L$,
we bound it by waiting until time $\tau$ and then using \eqref{H-pi-Z}.
Combining with \eqref{HA*}, as $\mean{\tau} = O(L) = o(n/|A|)$,
we deduce $p_{A,L} \le 2\mu_2 - 1 + o(1)$.
\end{proof}

The next lemma implements the second part of the above sketch,
in which we consider the probability that a fixed particle starting at some $v$ in $A$
hits some other particle before either has returned to $A$. For later use in analysing
two-type annihilation, we prove a more general lemma in which we only consider collisions
with particles starting in some set $B$ (this will correspond later to the other colour).

\begin{defn} \label{pvBT}
Given $B\subset V(G)$, $v \in V(G) \setminus B$ and a time $T$,
we let $p_{v,B,T}$ be the probability that a lazy random walk on $G$
starting from $v$ visits $B$ before time $T$.
\end{defn}

\begin{lem}\label{lem:meeting-vs-return}
Let $C>0$, $x \in (0,1/2)$ and $A\subset V(G)$ with $n^x<k=|A|<n^{1-x}$. 
Fix $v \in A$ and $B\subseteq A\setminus \{v\}$. 
One walker starts at each vertex of $A$ and at each time step 
a randomly selected walker takes a lazy random walk step on $G$.
For each $w\in B$, let $E_{w}$ be the event that the walkers started from $v$ and $w$
collide within $Tk$ steps, where $T := C\log n$, 
with the former not having reached $B$ and the latter not having returned to $B$.
Let $X = \sum_{w \in B} 1_{E_w}$. Then $\mean{X}\leq (1+o(1))p_{v,B,3T}$.
\end{lem}

\begin{proof}
As in Section \ref{sec:lonely}, it will be convenient to consider a Poissonised model,
where each particle independently takes $\Po(1/k)$ lazy random walk steps in each time interval.
Then with high probability there are $Tk$ movements within $(1+o(1))Tk$ steps
and each particle takes $(1+o(1))T$ steps.
Let $E'_w$ be the event that the particle started at $v$ reaches $B$ within $3T$ time intervals, 
with $w$ being the first vertex of $B$ it reaches. Clearly $p_{v,B,3T}=\sum_{w\in B}\prob{E'_w}$, 
and so it is sufficient to prove that $\prob{E_w}\leq (1+o(1))\prob{E'_w}$ for each $w \in B$.
Here we can replace $E_w$ by the same event with the additional restriction 
that neither particle moves away from the collision site in the same time interval 
that the collision occurs,
\ie we may assume that the two particles coincide at some integer time.
Indeed, it is not hard to see that this change affects the overall probability by a factor $1+O(1/k)=1+o(1)$.

Now we will implement the trajectory reversal strategy mentioned in the proof sketch.
This will let us compare $\prob{E_w}$ with $\prob{E'_w}$ indirectly 
via the following closely related `occupancy variables'.
We relate $E_w$ to the variable $Y$, defined by considering the first $1.5T$ time intervals, 
setting $Y=0$ if there is no collision satisfying $E_w$, or otherwise considering the first such collision,
and letting $Y$ be the number of time intervals that both particles remain at the collision site.
We claim that $\mean{Y}=(1-o(1))k\prob{E_w}$. This holds,
as if a collision satisfying $E_w$ occurs, then it does so within time $T$, 
and within the remaining time
period of $Y$ the expected time before either particle leaves is $(1-o(1))k$, 
as the rate of either particle leaving is $(1/2)(1/k + 1/k)$ by Poisson splitting and combining.
Similarly, we define $Y'$ with $\mean{Y'}=(1-o(1))k\prob{E'_w}$ by considering 
the first $3T$ time intervals, 
setting $Y'=0$ if the walker started at $v$ does not first hit $B$ at $w$,
or otherwise letting $Y'$ be the number of time intervals ending at an even time 
that the walker remains at $w$ after this first hit.

Now consider any trajectories $S$, $S'$ for the particles started at $v$, $w$
that lead to some time interval $[t-1,t]$ being counted in the random variable $Y$.
Following $S$ and then the reverse of $S'$ gives a trajectory $S''$
with $\prob{S''}=\prob{S \& S'}$ that leads to $[2t-1,2t]$ being counted for $Y'$.
Summing over such trajectories we deduce $\mean{Y}\leq\mean{Y'}$,
and so $\prob{E_w}\leq (1+o(1))\prob{E'_w}$, as required.
\end{proof}

We conclude this section with the proof 
of the lower bound for one-type annihilation.

\begin{proof}[Proof of Theorem \ref{thm:lower-bound-one}]
Let $G$ be a regular graph on $n$ vertices with spectral gap more than $1/3$,
\ie $\mu_2 = 2/3 - c$ for some $c>0$. 
Consider one-type annihilation on $G$ using lazy random walks
from any starting configuration and let $T$ be the extinction time.
We fix $x\in(0,1/2)$ and consider the regime while the number $k$ of remaining particles 
satisfies $n^{1-x}\geq k\geq n^x$. 
As discussed at the beginning of the section, it suffices to show the following `key step':
in this regime there are constants $c_1,c_2,c_3>0$
so that starting from any configuration $A$ of $k$ particles, 
with probability at least $c_1$ we need at least $c_2 n$ steps
to reduce the number of particles to $c_3 k$.

Following the sketch at the beginning of the section, we consider a mixing phase
of $\Theta(k\log n)$ steps (with the usual coupling of true and fake particles).
With high probability each fake particle takes $\Theta(\log n)$ steps in this phase.
By Lemma \ref{lem:no-return} the expected number of such particles
that return to $A$ is at most $(2\mu_2-1+o(1))k = (1/3 - 2c + o(1))k$.
The same bound applies to the number of particles annihilated by such particles.
By Lemma \ref{lem:meeting-vs-return}, taking $B=A\setminus\{v\}$ and summing over $v$,
the same bound applies to the expected number of such particles that collide 
with another particle before either has returned to $A$. Excluding these three sets
leaves a set of particles surviving the mixing phase with expected size at least $(6c-o(1))k$.
Letting $E_1$ be the event that at least $2ck$ particles survive the mixing phase,
we have $\prob{E_1} > 2c$.

After the mixing phase, at each time step we bound the probability of a collision of true particles
by the probability of a collision of fake particles, which is at most $2k/n$, say. 
We let $E_2$ be the event that there are at least $ck$ such collisions
in the $c^2 n/2$ steps following the mixing phase. 
Then $\prob{E_2} \le c$ by Markov's inequality, so $\prob{E_1 \cap E_2^c} > c$.
On this event, at least $ck$ true particles have survived, so in the $c^2 n/2$ steps
following the mixing phase, by Chernoff, with high probability 
there are at least $c^3 n/4$ steps by true particles. Thus we have 
established the required key step with $c_1=c_3=c$ and $c_2=c^3/4$.
\end{proof}

\subsection{Two-type annihilation}

Now we consider the lower bound for two-type annihilation, 
which introduces several complications not seen in the one-type model.
Firstly, we need to adapt the previous argument to allow for colours of differing speeds.
Secondly, the potential to have multiple particles at the same location causes two issues:
\begin{enumerate}[nosep] 
\item The events that $a$ first returns to $A$ at the location of a given other particle $b$ 
are not disjoint for two different choices of $b$ at the same vertex,
\item Our previous argument for particles escaping from $A$ works for a particle chosen
from a uniformly random occupied site, so does not say much about escape of a uniformly
random particle if the particles are concentrated at `bad' sites in $A$.
\end{enumerate}

Our approach is to first consider \emph{representative} walks 
obtained by starting a single particle at each occupied site.
The argument from the one-type model can be adapted
to show that with high probability a positive proportion of these representatives 
are unlikely to encounter any opposite-coloured representatives in $\Omega(n)$ steps. 
The key idea is that we can then use Reimer's inequality \cite{Rei00}
to argue that when the additional particles on multiply-occupied sites are taken into account, 
an $n^{-o(1)}$ proportion of these particles still survive.
While we cannot afford repeated loss of an $n^{-o(1)}$ factor,
we then show that the surviving particles are sufficiently well-behaved 
that we can thereafter follow the approach used for the one-type process.

\subsubsection{Very slow reds}

Our argument requires an  adaptation of the trajectory reversal argument 
in Lemma \ref{lem:meeting-vs-return} which breaks down when the red particles are very slow.
We therefore start by disposing of the case $p \le n^{-2/3}$, which is not hard to handle more directly,
and in fact our result here is stronger, as no quantitative assumption about expansion is needed:
it applies to any expander sequence. 

The first observation for this regime is that with high probability at most $n^{1/3}\log^2 n$ red movements occur
during the first $n\log n$ steps. We consider stages where $k=n^c$ particles of each colour remain with $c \ge 1/2$.
If we have not yet taken enough steps for the required lower bound 
then the red particles occupy at least $(1-n^{-0.1})k$ distinct vertices.
The following lemma will therefore complete the proof for this regime.

\begin{lem} \label{lem:slowreds}
There exist constants $c_1,c_2,c_3$ and $n_0$ (depending only on $\mu_2(G)$) with the following property. 
Let $n\geq n_0$ and $p\leq n^{-2/3}$. Fix a starting state with $k=n^c$ particles of each colour, where $1/2\leq c\leq 3/4$, 
arranged arbitrarily subject to having at least $(1-1/\log n)k$ vertices occupied by red particles. 
Then with probability at least $c_1$, at least $c_2n$ steps are required to reduce to $c_3k$ particles of each colour.
\end{lem}

The proof of Lemma \ref{lem:slowreds} uses the following fact about independent Bernoulli random variables:
if the expected number of successes is bounded, but the probability of at least one success is high, then
there must be some individual trial that has a high probability of success.
\begin{lem}\label{bernoulli-avoidance}Fix $x\in(0,1)$ and let $p_1,p_2,\ldots$ be any sequence satisfying $0\leq p_i\leq x$
for each $i$ and $\sum_ip_i\leq y$. Then $\prod_i(1-p_i)\geq (1-x)^{y/x}$.
\end{lem}
\begin{proof}
It is sufficient to show for any fixed $\ell\geq y/x$ that $\prod_{i=1}^\ell(1-p_i) \geq (1-x)^{y/x}$.
Given $\ell$, by compactness, we can consider a sequence $p_1,p_2,\ldots$ that minimises  $\prod_{i=1}^\ell(1-p_i)$.
Clearly $\sum_{i=1}^\ell p_i=y$. Observe that if $0<p_i\leq p_j<x$ for some $i\neq j\in [\ell]$ then
we can increase the product by replacing $(p_i,p_j)$ by $(p_i-\delta, p_j+\delta)$.
Thus $p_1,p_2,\ldots$ can have at most one term not equal to $0$ or $x$.
Writing $\alpha = y/x-\floor{y/x}$ we can assume that $p_i=x$ for $i\leq \floor{y/x}$,
$p_{\floor{y/x}+1}=\alpha x$ and $p_i=0$ for $i>\floor{y/x}+1$.
Since $1-\alpha x \geq (1-x)^\alpha$ for $x,\alpha\in [0,1]$, the result follows.
\end{proof}

\begin{proof}[Proof of Lemma \ref{lem:slowreds}]
With high probability at most $k/\log n$ red movements occur in the next $n$ steps, so we may assume
at least $k-2k/\log n$ red particles start at different vertices and do not move. As in the proof for one-type annihilation,
it is sufficient to show that with constant probability a constant proportion of these survive the mixing phase.

The expansion condition implies that we may choose some constant $q$, depending only on the spectral gap,
such that at least $2/3$ of the occupied vertices are `good' in that a walker starting from that vertex 
will return to the occupied set within the mixing phase with probability at most $q$.
If a constant fraction of blues are in good spots, we are done, since this implies a constant fraction
of the blue particles, and hence a constant fraction of the red, will survive.
Otherwise almost all the blues are in bad spots, and there are at least twice as many good reds as
blue spots.

For each blue spot, there is at most one good red spot which has more than $1/2$ probability of being 
the first one hit. Each other good spot has at most probability $(1+q)/2$ of being hit at all -- since it has
at most probability $q$ of being hit given that it is not the first one hit. Moreover, the total over all
good red spots of the probability of that spot being hit is at most $(1-q)^{-1}$, since this is a bound on the 
expected number of times a particle hits a good spot.

In particular, there are at least $k/2$ good red spots such that no individual blue has more than $(1+q)/2$
probability of reaching that spot. Note that $(1-q)^{-1}$ is an upper bound on the expected number of good red
spots hit by any given blue particle. Thus at most $k/4$ good red spots will be hit by more than $4(1-q)^{-1}$
blue particles in expectation.

Fix a red spot, and let $(p_i)_{i=1}^k$ be the probabilities of each blue particle reaching that spot.
By the analysis above, there are at least $k/4$ red spots with the property that $\max_ip_i\leq (1+q)/2$ and
$\sum_ip_i\leq 4(1-q)^{-1}$. It follows from Lemma \ref{bernoulli-avoidance} that each such red has constant 
probability of avoiding being hit. Thus, by Markov's inequality, a constant fraction of red particles survive 
the mixing phase with constant probability.
\end{proof}

\subsubsection{Reds are not too slow}

Henceforth we can assume $p \ge n^{-2/3}$.
Let $A$ be the occupied set when $k=n^c$ particles of each colour remain, for $c$ in some suitable range.
Consider starting one representative lazy random walk per site. By the proof of Lemma \ref{lem:no-return},
the probability of a uniformly random representative non-trivially returning to $A$ 
before taking $T := \eps n/k$ steps is $p_{A,T} \leq 2\mu_2-1+o_\eps(1)$. 
We fix some parameter $q$ to be optimised later. 
We say that $v \in A$ is \emph{good} if $p_{v,A,T} \le 1-q$. 
Thus a good representative has probability at least $q$ of not returning  non-trivially to $A$ within time $T$.
As $p_{A,T} = \mathbb{E}_v (p_{v,A,T})$, at least $\sigma k - o(k)$ vertices are good,
where $(1-\sigma)(1-q) = 2\mu_2-1$. 

Next we consider an optimisation problem according to the distribution of colours at good vertices.
Suppose that $A$ has $x_i k$ sites of colour $i$, of which $\sigma_i k$ are good, 
for $i=0,1$, labelling `red' as $0$ and `blue' as $1$. 
Then $x_0 + x_1 = 1$ and $\sigma_0+\sigma_1 \ge \sigma - o(1)$,
so we can fix $i$ with $\sigma_i / x_{1-i} \ge \sigma$. 
We can assume $x_{1-i}=\Omega(1)$, as otherwise we can choose $1-i$ instead.
The easier case is $i=1$ (blue is good), so we will  consider the case $i=0$ (red is good).
Our next lemma is an adaptation of Lemma \ref{lem:meeting-vs-return}.
The intended application is that $v$ is a good red representative 
and $B$ is the set of blue representatives.
The conclusion is that if the reds are not too slow then we have the analogue
of the second part of our argument in the one-type case. We may in particular assume $p\geq n^{-2/3}$,
and so the condition below will easily hold provided $c \leq 1/4$.

\begin{lem}\label{lem:mr2}
Let $c,\eps>0$, $p \in (0,1)$ with $p > \eps^{-2} k/n$, 
$v \in V(G)$ and $B\subset V(G) \setminus \{v\}$ with $k=|B|=n^c$.
One walker starts at each vertex of $B \cup \{v\}$.
The walkers independently take lazy random walk steps on $G$,
at rate $p/k$ for $v$ and $(1-p)/k$ for $B$.
For each $w\in B$, let $E_{w}$ be the event that the walkers 
started from $v$ and $w$ collide within time $pkT$ with $T := \eps n/k$, 
with the former not having reached $B$ and the latter not having returned to $B$.
Let $X = \sum_{w \in B} 1_{E_w}$. Then $\mean{X}\leq (1+o_\eps(1))p_{v,B,3T}$.
\end{lem}

\begin{proof}
As in Lemma \ref{lem:meeting-vs-return}, it suffices to show
$\prob{E_w}\leq (1+o(1))\prob{E'_w}$ for each $w \in B$,
where $E'_w$ is the event that the particle started at $v$ reaches $B$ within $3T$ steps, 
with $w$ being the first vertex of $B$ it reaches.
Instead of discretising time, we relate $E_w$ to the variable $Y$, 
defined by considering continuous time up to time $1.5pkT$, 
setting $Y=0$ if there is no collision satisfying $E_w$, or otherwise considering the first such collision
and letting $Y$ be the length of time that both particles remain at the collision site.
Then $\mean{Y}=(1-o_\eps(1))2k\prob{E_w}$, as $pkT > \eps^{-1} k$
and the rate of either particle leaving is $(1/2)(p/k + (1-p)/k)=1/2k$.
Similarly, we define $Y'$ with $\mean{Y'}=(1-o(1)) 2p^{-1} k \prob{E'_w}$  
by considering  continuous time up to time $3kT = 3\eps n$, 
setting $Y'=0$ if the walker started at $v$ does not first hit $B$ at $w$,
or otherwise letting $Y'$ be the time that the walker remains at $w$ after this first hit.

Now consider any trajectories $S$, $S'$ for the particles started at $v$, $w$
that lead to some time segment $(t,t+dt)$ being counted in the random variable $Y$.
Let $S''$ range over the trajectories obtained by following $S$ to some time in $(t,t+dt)$
and then the reverse of $S'$ with a time change factor $\frac{1-p}{p}$.
Then $\prob{S''}=\prob{S \& S'}$ and $S''$ leads to counting $p^{-1} (t,t+dt)$ in $Y'$.
Integrating over such trajectories we deduce $p^{-1}\mean{Y}\leq\mean{Y'}$,
and so $\prob{E_w}\leq (1+o(1))\prob{E'_w}$, as required.
\end{proof}

For $p \ge n^{-2/3}$ and $c \le 1/4$ we can apply Lemma \ref{lem:mr2},
so each good red representative $v$ has probability at least $2q-1$ of not returning to $A$
and not hitting a blue representative that has not returned to the blue set $B$.

It remains to control meetings with the set $B'$ of blue representatives that have returned to $B$.
By Lemma \ref{lem:no-return} applied to $B$, we have $\mean{B'} \le (2\mu_2-1+o_\eps(1))|B|$.
For each $w \in B$, similarly to the proof of Lemma \ref{lem:mr2} we can bound the expected number
of meetings of the $w$-representative with the $v$-representative by the expected number of visits
by the $w$-representative to $v$, which is at most $1/q$ as $v$ is good. 
Recalling that $\sigma_0 / x_1 \ge \sigma$, 
the expected number of $v$ not meeting such $w$ is at least
\[ (2q-1)\sigma_0 k -  (2\mu_2-1+o_\eps(1)) x_1 k/q 
\ge \sigma_0 k \left( 2q-1 - \frac{2\mu_2-1+o_\eps(1)}{q\sigma} \right) = \Omega(k), \]
provided that we can choose $q$ so that 
\[ (2q-1)q\sigma  > 2\mu_2 - 1, \text{ where } (1-\sigma)(1-q) = 2\mu_2-1. \]
Some calculations (we omit the details) show that this is possible if $\mu_2 < 0.575$.

We therefore expect some positive proportion $\eta$ of the good red representatives 
to avoid collisions with blue representatives. We say that a good red representative
is \emph{excellent} if has probability at least $\eta/2$ of avoiding such collisions.
Then at least $\eta/2$ proportion of the good red representatives are excellent.

Now we consider the multiplicities of true particles for each representative;
with high probability these are at most $8\log n/\log\log n$ by Lemma \ref{lem:poisson2}.
We lower bound the survival probability of each excellent red representative
by considering $8\log n/\log\log n$ independent samples of the blue representatives.
The survival probability is thus at least $\exp(-\Omega(\log n/\log\log n))=n^{-o(1)}$,
so we expect $n^{c-o(1)}$ particles to survive the mixing phase.
If they do, they are a subset of $n^c$ mixed particles, and with high probability 
no ten (say) of these particles occupy the same vertex in the next $n\log n$ steps. 
Thus we may thereafter reapply the same analysis revealing only ten samples of blue trajectories,
which gives constant probability of a constant fraction surviving the mixing phase,
as required to complete the proof as in the case of one-type annihilation. 

It remains to convert the above expected survival of $n^{c-o(1)}$ particles 
to a high probability statement. We need to bound the probability that too many
red representatives are destroyed, given that each is destroyed with probability
at most $1-\alpha$ with $\alpha = n^{-o(1)}$ by the samples of the blue representatives.
For a given set of red representatives to be destroyed, it must be possible to assign them 
to meets with distinct blue representatives (possibly from the same vertex in different samples).
This corresponds to events occurring \emph{disjointly} in the sense of Reimer's inequality  \cite{Rei00}, 
\ie each annihilation event is determined by the outcomes of some set of coordinates in the total
product probability space such that all these sets are disjoint for different annihilation events.
Thus the probability of destroying a given set of $t$ red representatives is at most $(1-\alpha)^t$,
so the number destroyed is stochastically dominated by a binomial with mean $(1-\alpha)n^c$. 
By Chernoff, with high probability $n^{c-o(1)}$ particles survive, as required.

\section*{Acknowledgements}
We are grateful to an anonymous referee for helpful comments on the exposition. Both authors were supported by ERC Advanced Grant 883810.


\begin{thebibliography}{99}

\bibitem{AF22} D. Ahlberg and C. Fransson,
\newblock Multi-colour competition with reinforcement.
\newblock \textit{Annales de l'Institut Henri Poincar\'e Probabilit\'es et Statistiques}, to appear.

\bibitem{AFbook} D. Aldous and J. A. Fill,
\newblock Reversible Markov chains and random walks on graphs. 
\newblock Unfinished monograph (2002, recompiled 2014).

\bibitem{Arr81}R. Arratia,
\newblock Limiting point processes for rescalings of coalescing and annihilating random walks on $Z^d$. 
\newblock \textit{Ann. Probab.} \textbf{9}:6 (1981), 909--936.

\bibitem{Arr83}R. Arratia,
\newblock Site recurrence for annihilating random walks on $Z^d$. 
\newblock \textit{Ann. Probab.} \textbf{11}:3 (1983), 706--713.

\bibitem{Bollobas} B. Bollob{\'a}s,
\newblock Modern graph theory.
\newblock Springer Science \& Business Media \textbf{184} (1998).

\bibitem{BG80}M. Bramson and D. Griffeath,
Asymptotics for interacting particle systems on $Z^d$. 
\newblock\textit{Z. Wahrsch. Verw. Gebiete} \textbf{53} (1980), 183--196.

\bibitem{BL90}M. Bramson and J. L.  Lebowitz,
\newblock Asymptotic behavior of densities in diffusion dominated two-particle reactions.
\newblock \textit{Phys. A} \textbf{168}:1 (1990), 88--94.

\bibitem{BL91}M. Bramson and J. L.  Lebowitz,
\newblock Spatial structure in diffusion-limited two-particle reactions.
\newblock \textit{J. Statist. Phys.} \textbf{65}:5--6 (1991), 941--951.

\bibitem{BL01}M. Bramson and J. L.  Lebowitz,
\newblock Spatial structure in low dimensions for diffusion limited two-particle reactions.
\newblock \textit{Ann. Appl. Probab.} \textbf{11}:1 (2001), 121--181.

\bibitem{CRS14}M. Cabezas, L. T. Rolla, and V. Sidoravicius,
\newblock Non-equilibrium phase transitions: activated random walks at criticality.
\newblock\textit{J. Stat. Phys.} \textbf{155}:6 (2014), 1112--1125.

\bibitem{CRS18}M. Cabezas, L. T. Rolla, and V. Sidoravicius,
\newblock Recurrence and density decay for diffusion-limited annihilating systems.
\newblock \textit{Probab. Theory Related Fields} \textbf{170}:3 (2018), 587--615.

\bibitem{2-type} I. Cristali, Y. Jiang, M. Junge, R. Kassem, D. Sivakoff, and G. York,
\newblock Two-type annihilating systems on the complete and star graph.
\newblock\textit{Stochastic Process. Appl.} \textbf{139} (2021), 321--342.

\bibitem{Poisson} C. Canonne,
\newblock A short note on Poisson tail bounds.
\newblock\url{http://www.cs.columbia.edu/~ccanonne/files/misc/2017-poissonconcentration.pdf} (2016).

\bibitem{CEOR} C. Cooper, R. Els\"{a}sser, H. Ono, and T. Radzik,
\newblock Coalescing random walks and voting on connected graphs.
\newblock \textit{SIAM J. Discrete Math.} \textbf{27}:4 (2013), 1748--1758.

\bibitem{CFR09}C. Cooper, A. Frieze and T. Radzik,
\newblock Multiple random walks in random regular graphs.
\newblock\textit{SIAM J. Discrete Math.} \textbf{23} (2009), 1738--1761.

\bibitem{CTW} D. Coppersmith, P. Tetali and P. Winkler,
\newblock Collisions among random walks on a graph.
\newblock \textit{SIAM J. Discrete Math.} \textbf{6}:3 (1993), 363--374.

\bibitem{deift2014random} P. Deift and P. Forrester,
\newblock Random Matrix Theory, Interacting Particle Systems and Integrable Systems (2014), 
Cambridge University Press.

\bibitem{DW83}P. Donnelly and D. Welsh,
\newblock Finite particle systems and infection models.
\newblock\textit{Math. Proc. Camb. Phil. Soc.} \textbf{94} (1983), 167--182.


\bibitem{Dur10}R. Durrett,
\newblock Some features of the spread of epidemics and information on a random graph. 
\newblock \textit{Proc. Natl. Acad. Sci. USA} \textbf{107} (2010), 4491--4498.

\bibitem{dygert2019bullet} B. Dygert, C. Kinzel, M. Junge, A. Raymond, E. Slivken and J. Zhu,
\newblock The bullet problem with discrete speeds.
\newblock \textit{Electron. Commun. Probab.} \textbf{24} (2019), 1--11.


\bibitem{EF85}Y. Elskens and H. L. Frisch, 
\newblock Annihilation kinetics in the one-dimensional ideal gas.
\newblock\textit{Phys. Rev. A}, \textbf{31} (1985), 3812--3816.

\bibitem{EN74} P. Erd\H{o}s and P. Ney, 
\newblock Some problems on random intervals and annihilating particles,
\newblock\textit{Ann. Probab.} \textbf{2}:5 (1974), 828--839.

\bibitem{EW17} P. Espanol and P. B. Warren,
\newblock Perspective: Dissipative particle dynamics.
\newblock\textit{J. Chem. Phys.}  \textbf{146}:150901 (2017).

\bibitem{clique} J. Haslegrave and P. Keevash,
\newblock Balanced two-type annihilation: mean-field asymptotics.
\newblock Manuscript, 2024. \url{https://arxiv.org/abs/2404.04128}

\bibitem{HP17} J. Haslegrave and M. Puljiz,
\newblock Reaching consensus on a connected graph.
\newblock \textit{Journal of Applied Probability} \textbf{54}:1 (2017), 181--198.

\bibitem{Her18}J. Hermon, 
\newblock Frogs on trees?
\newblock \textit{Electron. J. Probab.} \textbf{23} (2018), Paper No. 17, 40pp.  

\bibitem{big-bang}J. Hermon, S. Li, D. Yao and L. Zhang,
\newblock Mean field behavior during the Big Bang regime for coalescing random walks.
\newblock \textit{Ann. Probab.} \textbf{50}:5 (2022), 1813--1884.

\bibitem{HJJ19}C. Hoffman, T. Johnson and M. Junge,
\newblock Cover time for the frog model on trees. 
\newblock \textit{Forum Math. Sigma} \textbf{7} (2019), Paper No. e41, 49pp.

\bibitem{hoory2006expander} S. Hoory, N. Linial and A. Wigderson, 
\newblock Expander graphs and their applications.
\newblock \textit{Bulletin of the American Mathematical Society},
\textbf{43}:4 (2006), 439--561.

\bibitem{KRL95}P. L. Krapivsky, S. Redner and F. Leyvraz,
\newblock Ballistic annihilation kinetics: The case of discrete velocity distributions.
\newblock \textit{Phys. Rev. E}, \textbf{51} (1995), 3977--3987.

\bibitem{LP} D. A. Levin and Y. Peres,
\newblock Markov chains and mixing times.
\newblock Amer. Math. Soc. \textbf{107} 2nd ed. (2017).

\bibitem{liggett1985interacting} T. M. Liggett,
\newblock Interacting Particle Systems (1985), Springer New York.

\bibitem{liggett1999stochastic} T. M. Liggett,
\newblock Stochastic interacting systems: contact, voter and exclusion processes (1999),
Springer Science \& Business Media.


\bibitem{Loo77}J.-C. Lootgieter, 
\newblock Probl\`emes de r\'ecurrence concernant des mouvements al\'eatoires 
de particules sur $Z$ avec destruction.
\newblock\textit{Ann. Inst. H. Poincar\'e Sect. B (N.S.)} \textbf{13}:2 (1977), 127--139.

\bibitem{Lov96} L. Lov\'{a}sz,
\newblock Random walks on graphs: a survey.
\newblock In \textit{Combinatorics, {P}aul {E}rd\H{o}s is eighty, {V}ol. 2} ({K}eszthely, 1993),
\newblock Bolyai Soc. Math. Stud. \textbf{2} (1996), 353--397.

\bibitem{Pem07} R. Pemantle,
\newblock A survey of random processes with reinforcement.
\newblock \textit{Probability Surveys} \textbf{4} (2007), 1--79.

\bibitem{Rei00} D. Reimer,
\newblock Proof of the van den Berg--Kesten conjecture,
\newblock \textit{Combin. Probab. Comput.} \textbf{9}:1 (2000), 27--32.

\bibitem{RSSS19} N. Rivera, T. Sauerwald, A. Stauffer and J. Sylvester,
\newblock The dispersion time of random walks on finite graphs.
\newblock In \textit{The 31st ACM Symposium on Parallelism in
Algorithms and Architectures} (2019), 103--113. 

\bibitem{RSS23}N. Rivera, T. Sauerwald and J. Sylvester,
\newblock Multiple Random Walks on Graphs: Mixing Few to Cover Many.
\newblock \textit{Combin. Probab. Comput.} (2023). 

\bibitem{RS12}L. T. Rolla, and V. Sidoravicius,
\newblock Absorbing-state phase transition for driven-dissipative stochastic dynamics on $Z$.
\newblock \textit{Invent. Math.} \textbf{188}:1 (2012), 127--150.
\end{thebibliography}
\end{document}